\documentclass[12 pt]{article}      
\usepackage{amssymb}
\usepackage{eucal}

\makeatletter

\begin{document}

\newtheorem{theorem}{Theorem}[section]
\newtheorem{prop}[theorem]{Proposition}
\newtheorem{defn}[theorem]{Definition}
\newtheorem{lemma}[theorem]{Lemma}
\newtheorem{coro}[theorem]{Corollary}
\newtheorem{prop-def}{Proposition-Definition}[section]
\newtheorem{claim}{Claim}[section]
\newtheorem{propprop}{Proposed Proposition}[section]
\newtheorem{conjecture}{Conjecture}
\newtheorem{ex}{Example}[section]
\newcommand{\nc}{\newcommand}

\nc{\ola}[1]{\stackrel{#1}{\longrightarrow}}
\nc{\uap}[1]{\uparrow \rlap{$\scriptstyle{#1}$}}
\nc{\vp}{\vspace{8cm}}
\nc{\proofend}{$\blacksquare$\vspace{0.3cm}} \nc{\wbox}{\proofend}
\nc{\modg}[1]{\!<\!\!{#1}\!\!>}
\nc{\intg}[1]{F_C(#1)} \nc{\lmodg}{\!<\!\!} \nc{\rmodg}{\!\!>\!}
\nc{\cpi}{\widehat{\Pi}} \nc{\sha}{{\mbox{\cyr X}}}
\nc{\shprc}{\shpr_c} \nc{\labs}{\mid\!} \nc{\rabs}{\!\mid}

\nc{\ann}{\mrm{ann}}
\nc{\Aut}{\mrm{Aut}}
\nc{\can}{\mrm{can}}
\nc{\colim}{\mrm{colim}} \nc{\Cont}{\mrm{Cont}}
\nc{\rchar}{\mrm{char}} \nc{\cok}{\mrm{coker}}
\nc{\dtf}{{R-{\rm tf}}} \nc{\dtor}{{R-{\rm tor}}}
\renewcommand{\det}{\mrm{det}}
\nc{\Div}{{\mrm Div}} \nc{\End}{\mrm{End}} \nc{\Ext}{\mrm{Ext}}
\nc{\Fil}{\mrm{Fil}} \nc{\Frob}{\mrm{Frob}} \nc{\Gal}{\mrm{Gal}}
\nc{\GL}{\mrm{GL}} \nc{\Hom}{\mrm{Hom}} \nc{\hsr}{\mrm{H}}
\nc{\hpol}{\mrm{HP}} \nc{\id}{\mrm{id}} \nc{\im}{\mrm{im}}
\nc{\incl}{\mrm{incl}} \nc{\length}{\mrm{length}}
\nc{\mchar}{\rm char} \nc{\mpart}{\mrm{part}} \nc{\ql}{{\QQ_\ell}}
\nc{\qp}{{\QQ_p}} \nc{\rank}{\mrm{rank}} \nc{\rcot}{\mrm{cot}}
\nc{\rdef}{\mrm{def}} \nc{\rdiv}{{\rm div}} \nc{\rtf}{{\rm tf}}
\nc{\rtor}{{\rm tor}} \nc{\res}{\mrm{res}} \nc{\SL}{\mrm{SL}}
\nc{\Spec}{\mrm{Spec}} \nc{\tor}{\mrm{tor}} \nc{\Tr}{\mrm{Tr}}
\nc{\tr}{\mrm{tr}}

\nc{\Avg}{\mathbf{Avg}} \nc{\Alg}{\mathbf{Alg}} \nc{\bfk}{{\bf k}}
\nc{\bfone}{{\bf 1}} \nc{\changes}{\marginpar{\bf Changes!}}
\nc{\detail}{\marginpar{\bf More detail}
    \noindent{\bf Need more detail!}
    \svp}
\nc{\Mon}{\mathbf{Mon}} \nc{\proof}{\noindent{\bf Proof: }}
\nc{\remark}{\noindent{\bf Remark: }}
\nc{\remarks}{\noindent{\bf Remarks: }} \nc{\Rep}{\mathbf{Rep}} \nc{\Rings}{\mathbf{Rings}}

\nc{\BA}{{\Bbb A}} \nc{\CC}{{\Bbb C}} \nc{\DD}{{\Bbb D}}
\nc{\EE}{{\Bbb E}} \nc{\FF}{{\Bbb F}} \nc{\GG}{{\Bbb G}}
\nc{\HH}{{\Bbb H}} \nc{\LL}{{\Bbb L}} \nc{\NN}{{\Bbb N}}
\nc{\QQ}{{\Bbb Q}} \nc{\RR}{{\Bbb R}} \nc{\TT}{{\Bbb T}}
\nc{\VV}{{\Bbb V}} \nc{\ZZ}{{\Bbb Z}}

\nc{\cala}{{\cal A}} \nc{\calc}{{\cal C}} \nc{\cald}{\mathcal{D}}
\nc{\cale}{{\cal E}} \nc{\calf}{{\cal F}} \nc{\calg}{{\cal G}}
\nc{\calh}{{\cal H}} \nc{\cali}{{\cal I}} \nc{\call}{{\cal L}}
\nc{\calm}{{\cal M}} \nc{\caln}{{\cal N}} \nc{\calo}{{\cal O}}
\nc{\calp}{{\cal P}} \nc{\calr}{{\cal R}} \nc{\calt}{{\cal T}}
\nc{\calw}{{\cal W}} \nc{\calx}{{\cal X}} \nc{\CA}{\mathcal{A}}

\nc{\fraka}{{\frak a}} \nc{\frakB}{{\frak B}}
\nc{\frakm}{{\frak m}} \nc{\frakp}{{\frak p}}


\nc{\bfn}{\NN} \nc{\guess}{\noindent{\bf{Guess: }}}
\nc{\question}{\noindent{\bf{Question: }}}
\nc{\example}{\noindent{\bf{Example: }}}
\nc{\idea}{\noindent{\bf{Idea: }}} \font\cyr=wncyr10

\nc{\delete}[1]{{}}

\pagenumbering{roman}

\pagestyle{empty}

\title{ {\bf TITLE PAGE \\ FOR THE PH.D. DEGREE}\\
{\ }\\
\bf AN ALGEBRAIC STUDY OF AVERAGING OPERATORS }
\author{\small By Weili Cao}

\date{August 30, 2000}
\maketitle
\newpage
\pagestyle{plain}

\newpage
\hspace {1.2 in} ABSTRACT OF THE THESIS

\vspace {.4 in}

\hspace{.8 in} An Algebraic Study of Averaging Operators

\hspace{1.8 in} By Weili Cao

\hspace {1.13 in} Thesis director: Professor Li Guo

\vspace {0.4 in}

A module endomorphism $f$ on  an algebra $A$ is called an
averaging operator if it satisfies  $f(xf(y)) = f(x)f(y)$ for any
$x, y\in A$. An algebra $A$ with an averaging operator $f$ is
called an averaging algebra. Averaging operators have been studied
for over one hundred years. We study averaging operators from an
algebraic point of view. In the first part, we construct free
averaging algebras on an algebra $A$ and on a set $X$, and free
objects for some subcategories of averaging algebras. Then we
study properties of these free objects and, as an application, we
discuss some decision problems of averaging algebras. In the
second part, we show how averaging operators induce Lie algebra
structures. We discuss conditions under which a Lie bracket
operation is induced by an averaging operator. Then we discuss
properties of these induced Lie algebra structures. Finally we
apply the results from this discussion in the study
of averaging operators.

\newpage

\noindent

\tableofcontents

\newpage

\pagestyle{myheadings}

 \pagenumbering{arabic}
 \setcounter{section}{0}

\section{Introduction}

In this section, we first give the definition and some examples of
averaging algebras, followed by a brief description of the history
of the study of averaging algebras, and the motivation and
approach of our study. Then we give a summary of main results of
our study on free averaging algebras and the induced Lie algebras.

\subsection {Definition}
Let $R$ be a commutative ring with identity element $1_{R}$, and
$A$ be a commutative $R$-algebra with identity element $1_{A}$.

\begin{defn}
 An $R$-module endomorphism $f$ on  $A$ is called an {\bf
averaging operator} if it satisfies the {\bf averaging identity} :
\begin{equation}
 f(xf(y))=f(x)f(y),\ x,\ y\in A.
\label{eq:lav1}
\end{equation}
The pair $(A, f)$ is called an {\bf averaging algebra} or an {\bf
averaging $R$-algebra. }
\end{defn}
So if $f$ is an averaging operator, then an element
 $a = f(y) \in f(A)$ acts like a scalar: $f(ax) = af(x)$.

About  one hundred years ago, in a famous paper of Reynolds on
turbulence theory {\bf \cite {Rey1}}, the operator
\[ g(x,t) \rightarrow  \bar {g}(x,t) = \lim_{T \rightarrow \infty}\frac{1}{2T} \int_{-T}^{T} g(x, t+ \tau) d
\tau \] was defined for ergodic flow. It is an averaging operator.

\vspace{0.3cm} We now give two more examples of averaging
operators.

\begin{ex}
Let $a \in A$ be a fixed element, define $f_{a}(x) = ax$ for all
 $x \in A$. Then $f_{a}$ is an averaging operator. Note $f_{0_A}$
 and $f_{1_{A}}$ are the zero map and identity map, respectively.
 The set $\{f_{a}: a \in A \}$ is an $R$-algebra isomorphic to
 $A$.
\end{ex}

The following example is due to Rota {\bf \cite {Rot3}}.

\begin{ex}
 Let $A$ be the algebra over the real numbers consisting
  of real valued measurable functions on a  measure space $(S, \Sigma , m)$
  which are integrable over every set of finite measure in $\Sigma$ ,  and let
  $\Sigma '$ be a totally
  $\sigma$-finite $\sigma$-subfield of $\Sigma$. If
  $\varphi: f \rightarrow f'$ is  the linear operator in $A$ which  maps a
  function $f$ into the Radon-Nikodym derivative $f'$ of $f$ relative to
  $\Sigma '$, then $\varphi$ is an averaging operator.
\end{ex}

\subsection{History of studies on averaging operators }
Averaging operators first appeared in the work of Reynolds in 1895
in connection with the theory of turbulence {\bf \cite{Rey1}}, and
have been studied by many authors under various contexts. The
following is a list of some of the major research activities in
this area.

\begin{itemize}
\item
 Kamp\'{e} de F\'{e}riet first
recognized the importance of studying the averaging operators and
Renolds operators, and began their study in a series of papers
stretching over a period of thirty years
in the first half of the 20th century {\bf \cite{Kam}}.
\item
According to Rota{\bf \cite {Rot3}}, Birkhoff did the first study
of averaging operators using the methods of functional analysis in
the 1950's.
\item
In 1954, Moy showed the
connection between averaging operators and conditional expectation
{\bf \cite {Moy}}.
\item
In 1958, Kelley proved that a positive and idempotent linear
operator $T$ defined on the Banach algebra $C\sb \infty(X)$ of
real valued continuous functions vanishing at $\infty$ on a
locally compact Hausdorff space $X$ is an averaging operator if
and only if the range of $T$ is a subalgebra of $C\sb \infty(X)$
{\bf \cite {Kel1}}.
\item
In 1962,
Brainerd considered the conditions under which an
averaging operator can be represented as an integration on a ring
of functions {\bf \cite {Bra1}}.
\item
In 1964, Rota proved that a continuous Reynolds operator on the
algebra
 $L_{\infty}(S, \Sigma , m)$ of bounded measurable functions on a measure space
 $(S, \Sigma, m)$ is an averaging operator if and only
 if it has closed range {\bf \cite
{Rot3}}.
\item
In 1968, Gamlen and Miller discussed spectrum and
resolvent sets of averaging operators on Banach algebras
{\bf\cite {MI1}}.
\item
In 1969, Umegaki discussed averaging operators
on $B^{*}$ algebras and their applications to information channels
{\bf \cite {Ume}}.

\item
In 1976, Bong found connections between Baxter operators and
averaging operators on complex Banach algebras {\bf \cite {Bong}}.
\item
In 1986, Huijsmans generalized the work of Kelley
to the case of $f$-algebras {\bf \cite {Hui1}}.
\item
In 1993, Scheffold studied Reynolds operators and  averaging
operators on semisimple $F$-Banach lattice algebras {\bf
\cite{Sch}}.
\item
In 1998, Triki showed that a positive contractive projection
on an Archimedean $f$-algebra is an averaging operator
{\bf \cite{Tri2}}.
\end{itemize}

\subsection{Motivation and approach of this study}
We are interested in studying averaging operators because
\begin{enumerate}
\item[(i)]
they have  applications in many areas of pure and applied
mathematics, such as theory of turbulence, probability,
function analysis, and information theory{ \bf \cite
{Rey1}}{\bf \cite {Moy}}{\bf \cite {Kil}}{\bf \cite {Tri1}}{\bf
\cite {Tri2}} {\bf \cite {Ume}},
\item[(ii)] they are closely related
to Reynolds operators, symmetric operators, and Baxter operators
{\bf \cite {Rot3}}{\bf \cite {MI1}}{\bf \cite {Bong}},  and
\item[(iii)]
they naturally induce Lie algebra structures, as we will show
later in this thesis.
\end{enumerate}

As far as we know, most studies on averaging operators have been
done for various special algebras, such as function spaces, Banach
algebras, and the topics and methods have been basically analytic.
In this thesis, we study averaging operators in the general
context and from an algebraic point of view.

The first part of this thesis discusses free averaging algebras.
Free objects are important in the algebraic study of algebraic
structures. Although free objects in general can be described
using the language of universal algebra~{\bf \cite{Co}}, explicit constructions
have proved very useful. A good example is that the explicit
construction of  free modules of finite rank leads directly to the
structure theorems for finitely generated modules. Another example
is the construction of free Baxter algebras. Recently Guo and
Keigher gave an explicit construction of free Baxter algebras by
using mixed shuffle products {\bf \cite {G-K1, G-K2}},
generalizing the works of Cartier {\bf \cite{Ca}} and Rota {\bf
\cite{Rot2}}. This construction has been used to study properties
of Baxter algebras {\bf \cite{Guo1,Guo2,AGKO}}. Motivated by the
construction of free Baxter algebras by Guo and Keigher, we will
give an explicit construction of free averaging algebras on
algebras, and free averaging algebras on sets. We will also
construct free objects for some subcategories of averaging
algebras. Then we will discuss properties of these free objects,
and their applications in the study of averaging algebras in
general.

In the second part of this thesis, we will show how Lie algebras
structures can be naturally induced by averaging operators. We
consider the conditions under which a Lie algebra structure can be
induced by an averaging operator. We discuss properties of these
induced Lie algebra structures. We then use the  results obtained
to study properties of averaging operators.

\subsection {Main results on free averaging algebras }
We need some definitions before we give a summary of our main
results on free averaging algebras.

Let $(A,f)$ be an averaging $R$-algebra. An ideal $I$ of the
$R$-algebra $A$ is said to be an averaging ideal of $(A,f)$ if
$f(I) \subseteq  I$. The quotient averaging algebra is the
averaging $R$-algebra $(A/I, \bar{f})$, where $\bar{f}$ is defined
by $ \bar{f} : a + I \rightarrow f(a) + I$ , for all  $ a \in A $.

 Let $(A,f)$ and $(B,g)$ be two averaging
$R$-algebras. An $R$-algebra homomorphism (isomorphism)
 $\varphi : A \rightarrow B$ is said to be  an
 {\bf  averaging homomorphism (isomorphism) } if
  $\varphi \circ f = g \circ \varphi$.

 Let $A$ be an $R$-algebra,
  $(F, f)$ an averaging $R$-algebra, and
 $ i : A \rightarrow  F $  an $R$-algebra homomorphism. Together with
$ i : A \rightarrow  F $,  $(F, f)$ is said to be a {\bf free
averaging $R$-algebra on $A$ } if for any  averaging $R$-algebra
$(B, g)$ and $R$-algebra homomorphism
 $\varphi : A \rightarrow B $,  there exists a unique averaging homomorphism
 $ \hat {\varphi} : (F, f) \rightarrow (B,g)$ such that
   $\varphi =  \hat {\varphi} \circ i$.

Let $X$ be a  set, $(F, f)$ an averaging $R$-algebra, and
 $i: X \rightarrow F $ a map.   Together with the map
   $i: X \rightarrow F $, $(F, f)$  is said to be a {\bf free averaging
    $R$-algebra on $X$} if
for any
  averaging $R$-algebra $(B, g)$ and any map
   $\eta : X \rightarrow B$,  there exists a unique averaging
   homomorphism
   \[ \hat \varphi : (F, f) \rightarrow (B, g) \]
   such that $\hat \varphi \circ i = \eta $.

An averaging algebra $(A,f)$  is called {\bf unitary}
 if $f(1_{A}) = 1_{A}$. An averaging algebra $(A,f)$ is called a {\bf
Reynolds-averaging algebra}  if $f$ also satisfies the {\bf
Reynolds identity }
\[ f(x)f(y) + f(f(x)f(y)) = f(xf(y)) + f(yf(x)). \]
Free unitary averaging algebras and free Reynolds-averaging
algebras can be defined in the similar way.

Now  we give a brief description of the  main results on free
averaging algebras we have obtained in this thesis. Details can be
found in the indicated theorems or propositions which will be
given in later sections.

{\bf (1) Free averaging algebra on an $R$-algebra $A$  (Theorem
\ref{thm1})}

Let $F_{A} = A \otimes S(A)$ be the tensor algebra, where $S(A)$
is the symmetric algebra of $A$.  Let
  $f_{A} : F_{A} \rightarrow F_{A}$ be the map defined by
  \[ f_{A}(\sum_{i}a_{i} \otimes s_{i}) = \sum_{i}(1_{A} \otimes
a_{i}s_{i}), \]where $a_i \in A$, $s_i \in S(A)$. Then
  $(F_{A}, f_{A})$ is a free averaging algebra on $A$ .

{\bf(2) Free averaging algebra on a set $X$ (Theorem
\ref{freeOnX}) } Let $\Theta(X)$ be the set of monomials of $R[X]$
with coefficient $1_R$,  $Y$ be a set indexed by  $\Theta(X)$,
i.e.
  \[ Y = \{y_{\theta}: \theta \in \Theta(X)\}.\]
Let $F_X$ be the polynomial algebra
 $R[X \cup Y]$.  Define an $R$-linear  operator
 $f_X : F_X \rightarrow F_X$ by $f_X(uv) = y_uv$ where $u \in
 \Theta(X)$ and  $v \in R[Y]$. $f_X$ is an averaging
 operator on $F_X$, and together with the inclusion map $i_X: X \rightarrow F_X$,
  $(F_X, f_X)$ is a free averaging $R$-algebra on $X$.

{\bf (3) Free unitary averaging algebras and free
Reynolds- averaging algebras on a set $X$  (Proposition
\ref{prop1})}

If $I_0$ is the averaging ideal of $(F_X, f_X)$ generated by
$y_{1_R} - 1_R \in F_X$, then the quotient averaging algebra
 $(F_X/I_0, \bar{f_0})$ is a free unitary averaging $R$-algebra on $X$.

Let $I_1$ be the averaging ideal of $(F_X, f_X)$ generated by the
elements in the form
 $yy'y_{1_R} - yy' \in F_X$, where $y, y' \in Y$,  then the quotient averaging algebra
 $(F_X/I_1, \bar{f_1})$ is a free Reynolds-averaging $R$-algebra on $X$.

{\bf (4) Induced homomorphisms between free averaging algebras
 (Theorem \ref{thm4})}

 An $R$-algebra homomorphism $\theta: A
\rightarrow B$ induces an averaging homomorphism
 $\hat{\theta}: (F_{A}, f_{A}) \rightarrow (F_{B}, f_{B})$ which can be described nicely
 in terms of $\theta$, and
 (i) $\theta$  is  injective if and only if $\hat{\theta}$ is
 injective;  (ii) $\theta$  is surjective if and only if $\hat{\theta}$ is
 surjective; therefore (iii) $\theta$  is an isomorphism if and only if $\hat{\theta}$ is
 an averaging isomorphism.

{\bf(5) Chain conditions of averaging ideals of free averaging
algebras (Theorem \ref{thm5} and  \ref{thm6})}

(i)   $(F_{X}, f_{X})$ is a noetherian
averaging algebra  if and only if $X = \emptyset$.

(ii)  If $(F_{A}, f_{A})$ is a noetherian
averaging algebra, then $A$ is a noetherian $R$-algebra; The
converse is not true.

{\bf(6) Decision problems of averaging algebras (Theorem
\ref{thm7}) }
 Let $E_{2}$ be a finite set of $R$-algebra identities
involving a function symbol $\bf{f}$ of arity 1, and $E_{1}$ be
one of the following sets of $R$-algebra identities
\begin{eqnarray*}
&& E_{a} = \{ {\bf{f}}(v_{1}{\bf{f}}(v_{2})) =
{\bf{f}}(v_{1}){\bf{f}}(v_{2}) \}
\\ && E_{ua} = \{ {\bf{f}}(v_{1}{\bf{f}}(v_{2})) =
{\bf{f}}(v_{1}){\bf{f}}(v_{2}), \;\; {\bf{f}}({\bf{1}}) = {\bf{1}}
\}
\\ && E_{ra} =\{ {\bf{f}}(v_{1}{\bf{f}}(v_{2})) =
{\bf{f}}(v_{1}){\bf{f}}(v_{2}), \;\;
{\bf{f}}({\bf{f}}(v_{1}){\bf{f}}(v_{2})) =
{\bf{f}}(v_{1}){\bf{f}}(v_{2}). \}
\end{eqnarray*}
Then it is decidable whether $E_{1}$ implies $E_{2}$.

\subsection {Main results on induced Lie algebras }
We now list the main results on  Lie algebra structures induced by
averaging operators. First recall that a {\bf Lie bracket
operation} on an $R$-module $A$ is a bilinear binary operation
\[ [,]: A\times A \to A \]
such that
\begin{enumerate}
\item
{\bf (Anticommutativity) } $[x,x]=0$ for all $x \in A$ , and
\item
{\bf (Jacobi identity)} $[[x,y],z]+[[y,z],x]+[[z,x],y]=0 $ for all
 $x, y, z \in A$.
\end{enumerate}
The pair $(A,[,])$ is called a {\bf Lie algebra}.

 Let $f$ be an averaging operator on an
$R$-algebra $A$, we can define a binary operation on $A$:
\[ [x,y]_{f} = xf(y) - yf(x) \;\;\; x,y \in A. \]

We now summarize our main results on the  Lie algebras induced by
averaging operators.

{\bf (1) Lie algebras  Induced by averaging operators (Theorem
\ref{thm8})}

 Let $f$ be an averaging operator on  an
$R$-algebra $A$, then $[,]_{f}$ is a Lie bracket operation on $A$.
We denote the induced Lie algebra $(A, [,]_f)$ by $A_f$.

{\bf (2) Conditions under which a Lie bracket operation is
induced by an averaging operator (Theorem \ref{thm9} and
\ref{thm11})}

 (i)  A Lie bracket operation $[,]$ on  an $R$-algebra $A$
 is induced by an averaging operator  if and only
 if for all $ x, y \in A$
  \[[x,y] = x[1_{A},y] + y[x,1_{A}]\] and
  there exists some $t \in A$, such that
      \[ ([1_A,a] +at)[x,y] = [([1_A,a] +at)x,y] \]
 holds for all $a, x, y \in A$.

 (ii)  If $R$ is a field, and $A$ is of finite dimension over $R$, then
a  Lie bracket operation on $A$ is induced by some
 averaging operator if and only if a certain system of
 linear equations has a solution.

{\bf (3) Solvability and nilpotency
    of $A_{f}$ (Proposition \ref{prop3}) }

Let $f$ be an averaging operator on $A$.
\begin{enumerate}
\item[(i)]
$A_{f}$ is solvable of length 2.
\item[(ii)]
$A_{f}$ is nilpotent if and only if
\[ f(A)^{k} \subseteq \{ a
\in A : a[A,A]=0 \}\] for some $k >0$.
\end{enumerate}

{\bf (4) Nilpotent  radical of $A_{f}$  (Proposition
\ref{NrProp})}

 Let $A$ be a domain, and $f$ be an averaging operator on  $A$.
 We have

 (i) If $\ker(f) =0$, then $A_{f}$ is the nilpotent radical of
 $A_f$.

 (ii) If $\ker(f) \neq 0$, then $ker(f)$ is the nilpotent radical of $A_f$.

{\bf (5) Eigenvalues and eigenvectors of $ad_{f}(a)$ (Proposition
\ref{EigProp1}) }

 Let $f$ be an averaging operator on an $R$-algebra $A$. For each
 element $a \in A$, define a linear operator $ad_f(a)$ on $A$:
 \[ ad_f(a) : x \rightarrow [a, x], \;\; x \in A. \]
If $R$ is a field and  $A$  is a domain, then for each nonzero
element $a \in A$, $ad_f(a)$ has at most one nonzero eigenvalue.

{\bf (6) Kernel of $f$ (Proposition~\ref{prop:ker1}
    and Proposition~\ref{prop:ker2})}

(i) Let $X$ be a set. Then $\ker f_X=[F_X,F_X]_{f_X}$.

 (ii) Let
$(A,f)$ be an averaging algebra and let $\varphi: F_X\to A$ be a
surjective averaging homomorphism. Then
\begin{eqnarray*}
ker(f) = [A,A]_f &\Leftrightarrow&
\varphi^{-1}(ker(f)) = ker(f_X) + ker(\varphi)\\
&\Leftrightarrow&
ker(\varphi) \cap f_X(F_X) = f_X(ker(\varphi)).
\end{eqnarray*}

\newpage

\section{Constructions of free averaging algebras}
First we give or recall some definitions and basic properties of
averaging algebras. Then we construct a free averaging algebra on
an algebra. By combining two free structures, i.e. a free algebra
on a set and a free averaging algebra on an algebra, we construct
a free averaging algebra on a set. Finally we construct free
unitary averaging algebras and free Reynolds-averaging algebras.

\subsection{Basic properties of averaging algebras }
Recall that any ring $R$ is commutative with identity element
$\bfone_R$, and any algebra $A$ over $R$ is also commutative with
identity element $\bfone_A$ unless explicitly indicated otherwise.

\begin{defn}
Let $f$ be an averaging operator on  an $R$-algebra $A$. An
$R$-subalgebra $A_{1}$ of $A$ is called an {\bf averaging
subalgebra} of the averaging algebra $(A,f)$ if it is invariant
under $f$, i.e.
 $f(A_{1}) \subseteq A_{1}$.
\end{defn}

Also recall the following definitions we have given before.

An ideal $I$ of $A$ is called an {\bf averaging ideal} of the
averaging algebra $(A,f)$  if it is invariant under $f$, i.e.
 $f(I) \subseteq I $.

For two averaging $R$-algebras $(A,f)$ and $(B,g)$, an $R$-algebra
homomorphism (isomorphism) $\varphi : A \rightarrow B$ is called
an {\bf averaging homomorphism (isomorphism)} if
 $\varphi \circ f = g \circ \varphi$.
The category of averaging operators on  $A$
is denoted by $\Avg(A)$.

\begin{prop}
Let $A$ be an $R$-algebra,  $f \in \Avg(A)$. Then

(i) $f(A)$ is  closed under the multiplication of $A$.

(ii) $f(A) \cdot ker(f)  {\subseteq} ker(f) $.
\end{prop}

\begin{proof}
(i)  Let $x, y \in f(A)$, then
 $x = f(a), \; y = f(b)$ for some
  $a, b \in A$. Hence $xy = f(a)f(b) = f(af(b)) \in f(A)$.

(ii) Let  $x = f(a)$ for some $a \in A$.
     If $b \in ker(f)$, then
     $f(xb) = f(bf(a)) = f(b)f(a) = 0$.
      Hence  $f(A) \cdot ker(f)  {\subseteq} ker(f) $. \wbox

\end{proof}

\begin{prop}
If $I$ is an averaging ideal of $(A,f)$, then  $f$ induces an
averaging operator $\bar{f}$ on  $A/I$ given by
\[ \bar{f}(a + I) = f(a) + I \]
and the canonical epimorphism
\[ \pi: (A, f) \rightarrow (A/I, \bar{f}) \]
is an averaging homomorphism.
\end{prop}
\begin{proof}
If $a + I = a' + I$, then
 $a - a' \in I$. Since $I$ is invariant under $f$,
  $f(a)-f(a') = f(a - a') \in I$, hence
  $f(a) + I = f(a') + I$, and $\bar{f}$ is well-defined.

For all $a, b \in A$, we have
\begin{eqnarray*}
&& \bar{f}((a + I)\bar{f}(a'+I)) = \bar{f}((a + I)(f(a') + I))
\\&& = \bar{f}(af(a') + I) = f(af(a')) + I
\\&& = f(a)f(a') + I = \bar{f}(a + I) \bar{f}(a' + I ).
\end{eqnarray*}
So $\bar{f}$ is an averaging operator on   $A/I$. Also
\begin{eqnarray*}
&& \pi \circ f (a) = \pi (f(a)) = f(a) + I
\\ && = \bar{f}(a + I) = \bar{f} \circ \pi (a),
\end{eqnarray*}
hence $\pi$ is an averaging homomorphism. \wbox

\end{proof}

Averaging operators are not closed under composition and addition
of functions.
\begin{ex}
Let $A$ be the set of complex numbers. Then $A$ is a two-dimensional
algebra over the real numbers. Define two linear operators $f$ and
$g$ on $A$ by
\[f(z) = Im(z) = b, \]
\[ g(z) = zi ,  \]
for all $z = a + bi \in A$.  Both $f$ and $g$ are averaging
operators.  If we denote $f+g$ and $g \circ f$ by $h_1$ and $h_2$
respectively.  We have
 $h_1(1h_1(1)) = 0$, $h_1(i)h_1(i) = -1$, $h_2(ih_2(i)) = 0$,
 $h_2(i)h_2(i) = -1$. Hence neither $f + g$ nor $g \circ f$ are averaging
operators.
 \label{ex_cpl}
\end{ex}

However, we  have the following theorem.
\begin{theorem}
Let $f$ and $g$ be averaging operators on  an $R$-algebra $A$.

(i) For any $r \in R$, $rf$ is an averaging operator on  $A$.

(ii) If  $f$ is bijective, then $f^{-1}$ is an averaging operator.

(iii) If $f \circ g$ = $g \circ f$ ,
  then $f \circ g$ is an averaging operator.

(iv) If $f(xg(y)) + g(xf(y)) = f(x)g(y) + g(x)f(y)$  holds
    for all $x, y \in A$, then $f + g$ is an averaging operator.
\label{avgOpr}
\end{theorem}

\begin{proof}
(i) Let $h = rf$, we know that $h$ is a $R$-module endomorphism on
$A$. For all
 $x, y \in A$, we have
 \[h(xh(y)) = r(f(xrf(y))) = r^{2}f(x)f(y) = h(x)h(y). \]

 (ii) Clearly $f^{-1}$ is a bijective $R$-module endomorphism on  $A$. Let $x, y \in A$.
Since $f$ is surjective, there exists $ b \in A$ such that
 $f(b) = f^{-1}(y)$. We have
 \begin{eqnarray*}
 && f(f^{-1}(xf^{-1}(y))) = xf^{-1}(y)
 \\&& = f(f^{-1}(x))f^{-1}(y) = f(f^{-1}(x))f(b)
 \\&& = f(f^{-1}(x)f(b)) = f(f^{-1}(x)f^{-1}(y)).
 \end{eqnarray*}
 Since $f$ is injective, we have
  $f^{-1}(xf^{-1}(y)) = f^{-1}(x)f^{-1}(y)$.

  (iii) Let $h =f \circ g$ and   $x, y \in A$.
   \begin{eqnarray*}
&& h(xh(y)) = f \circ g (xf(g(y))) = f \circ g (xg(f(y)))
\\&& = f ( g (xg(f(y)))) = f ( g (x)g(f(y))) = f ( g (x)f(g(y)))
\\&& =f ( g (x))f(g(y)) = h(x) h(y).
\end{eqnarray*}

(iv)  Let $h =f + g$ and   $x, y \in A$.
  \begin{eqnarray*}
&& h(xh(y)) = (f + g) (x(f+g)(y)) = (f + g) (xf(y) + xg(y))
\\&& = f (xf(y)) + f(xg(y)) + g(xf(y)) + g(xg(y))
\\&& = f (x)f(y) + f(x)g(y) + g(x)f(y) + g(x)g(y)
\\&& = (f + g)(x)(f + g)(y) = h(x)h(y).
\end{eqnarray*}
We have proved the theorem. \wbox
\end{proof}

\begin{coro}
If $f$ is an averaging operator on  an $R$-algebra $A$, then for
any polynomial  $P(t) \in R[t]$ which does not have a constant
term, $P(f)$ is an averaging operator on  $A$.
\end{coro}
\begin{proof}
  For any $k \geq 1$,
   $f^{k+1} = f^{k} \circ f = f \circ f^{k}$. Hence that $f^{k}$
   is an averaging operator implies that $f^{k+1}$ is also an
   averaging operator. By induction we know that $f^{n}$ is an
   averaging operator for any $n \geq 1$.

   Now let
    $P(t)=P_n(t) = r_nt^{n} + r_{n-1}t^{n-1} + ... + r_1t \in R[t]$
    be a polynomial of degree $n$.
     We will show by induction that
       $P_n(f)=  r_nf^{n} + r_{n-1}f^{n-1} + ... + r_1f$ is an averaging operator.
       For $n = 1$, $P_1(f) = r_1f$ is clearly an averaging
       operator.  Assume  it has been established that for
       polynomial $P_k(t) = r_{k}t^{k} + ... + r_1t$  of degree  $k \geq
       1$,
       $P_k(f)$ is an averaging operator.  Denote $P_k(f)$ and $r_{k+1}f^{k+1}$ by
       $h_1$ and $h_2$ respectively. Then  for all
         $x, y \in A$, we have
\begin{eqnarray*}
&& h_1(xh_2(y)) = h_1(xr_{k+1}f^{k+1}(y))
 \\ && = \Sigma_1^{k}r_mf^{m}(xr_{k+1}f^{k+1}(y))
 \\&& = \Sigma_1^{k}(r_mf^{m}(x)r_{k+1}f^{k+1}(y))
 \\&& = (\Sigma_1^{k}r_mf^{m}(x))r_{k+1}f^{k+1}(y)
 \\&& =  h_1(x)h_2(y),
\end{eqnarray*}
and
\begin{eqnarray*}
&& h_2(xh_1(y)) = h_2(x \Sigma_1^{k}r_mf^{m}(y))
\\ && =  \Sigma_1^{k}h_2(xr_mf^{m}(y))
\\ && = \Sigma_1^{k}r_{k+1}f^{k+1}(xr_mf^{m}(y))
\\ && = \Sigma_1^{k}r_{k+1}f^{k+1}(x)r_mf^{m}(y)
\\ && = r_{k+1}f^{k+1}(x)\Sigma_1^{k}r_mf^{m}(y)
\\ && = h_2(x)h_1(y)
\end{eqnarray*}
       According to (iv) of  Theorem {\ref{avgOpr}},
       $P_{k+1}(f) = h_1 + h_2 $ is an averaging operator. \wbox
\end{proof}

Note that it is necessary to require that the polynomial  $P(t)$
does not have a constant term. For example, if $f$ is the
averaging operator in Example {\ref{ex_cpl}}, take  $P(t)= t + 1$.
Then $P(f) = f + id_A$ is not an averaging operator (To verify
this, take $x = y = i$).

\subsection{Free averaging algebras on an algebra}

Let $A$ be an $R$-algebra, $(F, f)$ an averaging $R$-algebra, and
 $ i : A \rightarrow  F $  an $R$-algebra homomorphism. Recall that
  $(F, f)$ is said to be a
free averaging $R$-algebra on $A$ if for any  averaging
$R$-algebra $(B, g)$ and $R$-algebra homomorphism
 $\varphi : A \rightarrow B $,  there exists a unique averaging homomorphism
 $ \hat {\varphi} : (F, f) \rightarrow (B,g)$ such that
   $\varphi =  \hat {\varphi} \circ i$.
The existence and uniqueness (up to isomorphism) follow from
general principles   of universal algebra.

Let $T(A)$ be the tensor algebra
\begin{eqnarray*}
&& \displaystyle T(A) = \bigoplus_{n \geq 0} T^{n}(A)
\\ && = T^{0}(A) \oplus T^{1}(A) \oplus ... \oplus
T^{n}(A) \oplus ...
\end{eqnarray*}
where
\begin{eqnarray*}
 && T^{0}(A) = R, \;\; T^{1}(A) = A,
\\ && T^{n}(A) = {\underbrace{A \otimes A \otimes ...\otimes A}_{n \;
factors}}\;\;\; for \; n \geq 2.
\end{eqnarray*}
$T(A)$ is an $R$-algebra with identity element $1_{T(A)} = 1_{R}$
and the multiplication is concatenation. Let $S(A)$ be the
symmetric algebra
\begin{eqnarray*}
&& S(A) = T(A)/H =  \bigoplus_{n \geq 0} S^{n}(A)
\\&& =S^{0}(A) \oplus S^{1}(A) \oplus ... \oplus
S^{n}(A) \oplus ...
\end{eqnarray*}
where $H$ is the ideal of $T(A)$ generated by
\[ \{ a_{1} \otimes a_{2} - a_{2} \otimes a_{1}: a_{1}, a_{2} \in A \} \]
and
\begin{eqnarray*}
&& S^{0}(A) = R, \;\; S^{1}(A) = A,
\\ && S^{n}(A) =  T^{n}(A)/(H \cap T^{n}(A)) \;\; \;
{\rm\ for\ } \; n \geq
2.
\end{eqnarray*}
$S(A)$ is a commutative $R$-algebra with identity element
$1_{S(A)} = 1_{R}$. We use
  $a_{1} \odot a_{2} \odot ... \odot a_{n}$
to denote the element of $S(A)$
  \[  a_{1} \otimes a_{2} \otimes ... \otimes a_{n} + H ,\]
hence an element of $S^{n}(A)$ can be written as
 \[ \displaystyle \sum_{i = 1}^{k}a_{i_{1}} \odot a_{i_{2}} \odot ... \odot
 a_{i_{n}} \]
for some $k > 0$.

Let $F_{A} = A \otimes S(A)$. Define a multiplication of $F_{A}$
by
\[\displaystyle (\sum_{i}a_{i} \otimes s_{i})(\sum_{j}a_{j}' \otimes
s_{j}')= \sum_{i,j}(a_{i}a_{j}') \otimes (s_{i}s_{j}'), \] where
$a_{i}, a_{j}' \in A$, and $s_{i}, s_{j}' \in S(A)$. $F_A$ becomes
an $R$-algebra.  Furthermore, we  define
\[ f_{A}: F_{A} \rightarrow F_{A}, \]
\[ f_{A}(\sum_{i}a_{i} \otimes s_{i}) = \sum_{i}(1_{A} \otimes
a_{i}s_{i}).\]
$f_{A}$ is well-defined since it is the linear
extension of the $R$-bilinear map
\[          \mu : A \times S(A) \rightarrow F_{A}, \]
\[           \mu (a, s) = 1_{A} \otimes (as). \]
$f_{A}$ is an averaging operator on  $F_{A}$,  since for
  $x, y \in F_{A},$
\[ x = \sum_{i}a_{i} \otimes s_{i}, \; y = \sum_{j}a_{j}' \otimes
s_{j}'\] we have
\begin{eqnarray*}
&& f_{A}(xf_{A}(y)) = f_{A}(\sum_{i,j}a_{i} \otimes
s_{i}a_{j}'s_{j}')
\\ && = \sum_{i,j}(1_{A} \otimes a_{i}
s_{i}a_{j}'s_{j}')
\\&& = (\sum_{i}1_{A} \otimes a_{i}
s_{i})( \sum_{j} 1_{A} \otimes a_{j}'s_{j}')
\\&& = f_{A}(x)f_{A}(y).
\end{eqnarray*}
Define $i_{A}: A \rightarrow F_{A}$ by
  $i_{A}(a) = a \otimes 1_{R}$ for
 $a \in A$. Then $i_A$ is an $R$-algebra homomorphism.

\begin{theorem}
Together with
 $i_{A}: A \rightarrow F_{A}$, $(F_{A}, f_{A})$ is a free averaging
 $R$-algebra on $A$.  In other words,
 for any  averaging $R$-algebra $(B, g)$ and $R$-algebra homomorphism
 $\varphi : A \rightarrow B $,  there exists a unique averaging homomorphism
 $ \hat {\varphi} : (F_{A}, f_{A}) \rightarrow (B,g)$ such that
   $\varphi =  \hat {\varphi} \circ i_{A}$.
  \label{thm1}
\end{theorem}
\begin{proof}
We define the map
 $ \hat {\varphi} : (F_{A}, f_{A}) \rightarrow (B,g)$  by
\[
 \hat {\varphi}( a\otimes s)=
\!\!\! \left \{ \begin{array}{cl}
 s \varphi(a),  & s \in S^{0}(A) = R, \\
\!\!\!\varphi(a)g( \varphi (a_{1}))...g( \varphi (a_{n})), &
\!\!\! \! s =
a_{1}\odot ... \odot a_{n}  \in S^{n}(A), n > 0. \end{array}
\right .
\]
We then extend $ \hat {\varphi}$ by $R$-linearity.

For each $a \in A$,
\[\hat {\varphi} \circ i_{A}(a) = \hat
 {\varphi}(a \otimes 1_{R})= 1_{R}\varphi(a) = \varphi(a). \]
Therefore $\hat {\varphi} \circ i_{A} = \varphi $.
To see ${\hat {\varphi}}$ is an averaging homomorphism, without
loss of generality, we take two elements in $F_{A}$:
 \begin{eqnarray*}
 && x = a \otimes (a_{1} \odot a_{2}\odot ...\odot a_{m}) \;\;
 \\ && y = a' \otimes (a_{1}' \odot a_{2}'\odot...\odot a_{n}')
 \end{eqnarray*}
 where $a, a', a_{i}, a_{j}'$ are elements of $A$, and we have
\begin{eqnarray*}
&& \hat{\varphi}(xy)
\\ && = \hat{\varphi}(aa' \otimes (a_{1} \odot ...\odot
a_{m}\odot a_{1}' \odot...\odot a_{n}'))
\\ && = \varphi(aa')g(\varphi(a_{1}))...g(\varphi(a_{m}))g(\varphi(a_{1}'))...g(\varphi(a_{n}'))
\\ && =( \varphi(a)g(\varphi(a_{1}))...g(\varphi(a_{m}))( \varphi(a')g(\varphi(a_{1}'))...g(\varphi(a_{n}'))
\\ && = \hat{\varphi}(x)\hat{\varphi}(y).
\end{eqnarray*}
We also have
\begin{eqnarray*}
&& (\hat {\varphi} \circ f_{A})(x)
\\ && = (\hat {\varphi} \circ f_{A})(a \otimes (a_{1} \odot ...\odot
a_{m}))
\\ && = \hat {\varphi}(1_{A} \otimes (a \odot a_{1}\odot...\odot
a_{m}))
\\ && = \varphi (1_{A})g(\varphi (a))g(\varphi (a_{1}))...g(\varphi (a_{m}))
\\ && = g(\varphi (a))g(\varphi (a_{1}))...g(\varphi (a_{m}))
\\ && = g(\varphi (a)g(\varphi (a_{1}))...g(\varphi (a_{m})))
\\ && = g(\hat {\varphi}(x)),
\end{eqnarray*}
therefore $\hat {\varphi} \circ f_{A} = g \circ \hat {\varphi}$,
and $\hat {\varphi}$ is an averaging homomorphism.

Now if $\psi :(F_{A}, f_{A}) \rightarrow (B,g)$ is another
averaging homomorphism satisfying $\psi \circ i_{A} = \varphi $,
then for $ a \in A$ and $r \in R$
 \begin{eqnarray*}
 && \psi (a \otimes r) = \psi(r(a \otimes 1_{R})) = r\psi(a \otimes 1_{R})
 \\ && = r(\psi \circ i_{A}(a)) = r \varphi(a) = \hat{\varphi} (a \otimes
 r),
\end{eqnarray*}
 and
\begin{eqnarray*}
 && \psi (a \otimes (a_{1} \odot  ... \odot a_{m}))
 \\ && = \psi((a \otimes 1_{R})(1_{A} \otimes a_{1})...(1_{A}
 \otimes a_{m}))
 \\ && = \psi(a \otimes 1_{R})\psi(1_{A} \otimes a_{1})...\psi(1_{A} \otimes a_{m})
 \\ && = \psi(a \otimes 1_{R})\psi(f_{A}(a_{1} \otimes 1_{R}))...\psi(f_{A}(a_{m} \otimes 1_{R}))
\\ && =  \psi(a \otimes 1_{R})g(\psi(a_{1} \otimes 1_{R}))...g(\psi(a_{m} \otimes 1_{R}))
 \\ && =
 \psi(i_{A}(a))g(\psi(i_{A}(a_{1})))...g(\psi(i_{A}(a_{m})))
 \\ && = \varphi(a)g(\varphi(a_{1}))...g(\varphi(a_{m}))
 \\ && = \hat {\varphi} (a \otimes (a_{1} \odot ... \odot a_{m}))
\end{eqnarray*}
Therefore $ \psi = \hat {\varphi}$. \wbox
\end{proof}

\subsection{Free averaging algebras on a  set}
Let $X$ be a  set, $(F, f)$ an averaging $R$-algebra, and
 $i: X \rightarrow F $ a map. Recall that  $(F, f)$ together with
 the map $i: X \rightarrow F $
  is said to be a free averaging $R$-algebra on $X$ if for any
  averaging $R$-algebra $(B, g)$ and any map
   $\eta : X \rightarrow B$,  there exists a unique averaging
   homomorphism
   \[ \hat \varphi : (F, f) \rightarrow (B, g) \]
   such that $\hat \varphi \circ i = \eta $.
   The existence and uniqueness (up to isomorphism)  of such free
   objects follow from the general principles   of universal algebra.

Let $j: X \rightarrow R[X]$ be the inclusion map from $X$ to the
polynomial algebra $R[X]$, $(F,f) = (F_{R[X]}, f_{R[X]})$, the
free averaging $R$-algebra on $R[X]$, as we constructed in last
subsection, and $i = i_{R[X]} \circ j : X \rightarrow  F$. We have
\begin{theorem}
Together with the map  $i: X \rightarrow  F$, the averaging
algebra
  $(F, f) = (F_{R[X]}, f_{R[X]})$ is a free averaging $R$-algebra
  on $X$.
  \label{thm2}
\end{theorem}
 \begin{proof}
  Let $(B,g)$ be an averaging $R$-algebra, and
   $\eta : X \rightarrow B$  a map.
  Since $ R[X]$ is a free $R$-algebra on $X$, there exists a
  unique $R$-algebra homomorphism
  \[ \varphi : R[X] \rightarrow B \]
  such that  $\varphi \circ j = \eta$.  In turn we have a unique
  averaging homomorphism
   \[\hat {\varphi} : (F_{R[X]}, f_{R[X]}) \rightarrow (B, g) \]
  such that $\hat {\varphi} \circ i_{R[X]} = \varphi$. Therefore
  \begin{eqnarray*}
  && \hat {\varphi} \circ i = \hat {\varphi} \circ (i_{R[X]} \circ j)
  \\ && = ( \hat {\varphi} \circ i_{R[X]}) \circ j
  \\ && = \varphi \circ j = \eta.
  \end{eqnarray*}

  If $\psi : F_{R[X]} \rightarrow B$ is another averaging
  homomorphism satisfying $\psi \circ i = \eta$, then
  $(\psi \circ i_{R[X]}) \circ j = \eta$. Since both $\psi$ and
  $i_{A}$ are $R$-algebra homomorphisms, $\psi \circ i_{R[X]}$ is an
  $R$-algebra homomorphism, and we have $\psi \circ i_{R[X]} = \varphi$
  due to the uniqueness of $\varphi$, which in turn implies
  $\psi = \hat {\varphi}$. \wbox
  \end{proof}

This construction is interesting in that it is a combination of
two free structures. The same approach works for other kinds of
free objects on a set such as free Reynolds algebras or free
Baxter algebras.

It is important to notice that $F_{R[X]}$ is a polynomial algebra.
If we use $\Theta(X)$ to denote  the set of nonzero monomials of
$R[X]$ with coefficient $1_R$, we can   define a set $Y$ which is
indexed by $\Theta(X)$ :
\[ Y = \{y_{\theta} : \theta \in \Theta(X) \}. \]
For any
   $\theta_1, \theta_2 \in  \Theta(X)$,  $y_{\theta_1} = y_{\theta_2}$ if and only if
  $\theta_1 = \theta_2$. Also note that
$X$ and $Y$ are disjoint. We have
\begin{eqnarray*}
&& S(R[X]) \cong R[Y], \;\; and
\\&& F_{R[X]} \cong R[X] \otimes R[Y] \cong R[X \cup Y ].
\end{eqnarray*}
We now give another description of the free averaging $R$-algebra
on the set $X$.

Let $F_X = R[X \cup Y]$. Define an $R$-endomorphism
   \[f_X : F_X \rightarrow F_X \]
by
   \[f_X(uv) = y_uv, \]
where $u \in \Theta(X)$, $v \in R[Y]$. We then extend the
definition of $f_X$ by $R$-linearity.

\begin{lemma}
The $R$-endomorphism $f_X$ is an averaging operator on $F_X$.
\end{lemma}
\begin{proof}
Take $ a = \sum_i u_iv_i$ and  $ b = \sum_js_jt_j \in F_X$, where
$u_i$ and $s_j \in \Theta(X)$,
 $v_i$ and  $t_j \in R[Y]$. Then we have
\begin{eqnarray*}
&& f_X(af_X(b))  = f_X((\Sigma_iu_iv_i)(\Sigma_jy_{s_j}t_j))
 \\&&  = f_X(\Sigma_i\Sigma_j(u_iv_iy_{s_j}t_j))= \Sigma_i\Sigma_j(y_{u_i}v_iy_{s_j}t_j)
\\&&  = (\Sigma_iy_{u_i}v_i)(\Sigma_jy_{s_j}t_j) =  f_X(a) f_X(b).
\end{eqnarray*}
Hence $f_X$ is an averaging operator. \wbox

\end{proof}

\begin{theorem}
Together with the inclusion map
 $i_{X}: X \rightarrow F_X$, the averaging $R$-algebra
 $(F_{X}, f_{X})$ is a free averaging $R$-algebra on $X$.
 \label{freeOnX}
\end{theorem}
\begin{proof}
Let $(B, g)$ be an averaging $R$-algebra,
 and $\eta : X \rightarrow B$ be a map. We need to show that there exists a unique
 averaging homomorphism $\hat {\varphi} : (F_X, f_X) \rightarrow (B,g)$, such that
 $\hat {\varphi} \circ i_X = \eta$.

  Let
 $\varphi : R[X] \rightarrow B$ be the unique $R$-algebra homomorphism
 induced by $\eta$, such that
 $\varphi \circ j = \eta$, where $j : X \rightarrow R[X]$ is the inclusion
 map.

  Define an $R$-algebra homomorphism
 \[ \hat {\varphi} : F_{X} \rightarrow B \]
by

(i) for each  $x \in X$, $\hat {\varphi}(x) = \eta (x)$,  and

(ii) for  each  $y_{\theta} \in Y$, where $\theta \in \Theta(X)$,
  $\hat {\varphi}(y_{\theta}) = g(\varphi(\theta))$.

Note that for $u \in \Theta(X)$,
 $\hat {\varphi}(u) = \varphi(u)$.

Since for each $x \in X$,
 \[ \hat {\varphi}\circ i_X (x) = \hat {\varphi}(x)  = \eta (x),\]
we have $\hat {\varphi} \circ i_X = \eta$.
 Now we verify that $\hat
{\varphi}$ is an averaging homomorphism, i.e.  $g \circ \hat
{\varphi}(a) = \hat {\varphi} \circ f_{X}(a)$ for all $ a \in
F_{X}$. We only need to check for the monomials with coefficient
$1_R$ in $F_X$. Take such a monomial
  $a = uv \in F_X$, where $u \in \Theta(X)$ and
  $v \in R[Y]$.

If $v = 1$, then
\begin{eqnarray*}
&& g \circ \hat {\varphi}(a) = g \circ \hat {\varphi}(u)
 \\ && = g(\hat {\varphi}(u)) = g(\varphi(u))
 \\ && = \hat {\varphi}(y_u) = \hat {\varphi} \circ f_X(u)
 \\ && = \hat {\varphi} \circ f_X(a).
\end{eqnarray*}

If $v = y_{\theta_1} ...y_{\theta_k}$ for some
 $\theta_1, ..., \theta_k \in \Theta(X)$, then
\begin{eqnarray*}
&& g \circ \hat {\varphi}(a) = g \circ \hat {\varphi}(uv)
 \\ && = g(\hat {\varphi}(u)\hat {\varphi}(y_{\theta_1})...\hat
 {\varphi}(y_{\theta_k}))
 \\ &&  = g(\varphi(u)g(\varphi(\theta_1))...g(\varphi(\theta_1)))
 \\ && =  g(\varphi(u))g(\varphi(\theta_1))...g(\varphi(\theta_1))
 \\ && =  \hat {\varphi}(y_u)  \hat {\varphi}(y_{\theta_1})...\hat {\varphi}(y_{\theta_k})
 \\ && = \hat {\varphi}(y_u y_{\theta_1} ... y_{\theta_k})
 \\ && = \hat {\varphi}(f_X(uy_{\theta_1} ... y_{\theta_k}))
 \\ && = \hat {\varphi}\circ f_X (uv) = \hat {\varphi}\circ f_X(a).
\end{eqnarray*}

If $\psi: (F_{X},f_{X}) \rightarrow (B,g)$ is also an averaging
homomorphism, and  $\psi \circ i_{X} = \eta $, then $\psi(\theta)
= \varphi(\theta)$ for all $\theta \in \Theta(X)$.

(i) for each  $x\in X$,
  $\psi(x)= \psi \circ i_{X}(x)=\eta(x)=\hat{\varphi}(x). $

(ii) for each $y_{\theta} \in Y$, where $\theta \in \Theta(X)$,
 $ \psi(y_{\theta}) = \psi (f_{X}(\theta)) = g(\psi (\theta)) = g(\varphi(\theta)) = \hat
 {\varphi}(y_{\theta}).$
Hence $\psi= \hat{\varphi}$ and  we have proved the theorem. \wbox
\end{proof}

\begin{coro}
A free averaging $R$-algebra on an $R$-algebra $A$ is (isomorphic
to) a polynomial algebra if $A$ is a polynomial algebra.
\end{coro}

Note that if $X = \emptyset$, then $\Theta(X) = \{1_R \}$ and
  $Y= \{y_{1_R} \}$.  The free averaging $R$-algebra on $X =
\emptyset$ is
  $(F_{\emptyset}, f_{\emptyset})$, where
  $F_{\emptyset} = R[y_{1_R}]$,  and $f_{\emptyset}(v) = y_{1_R}v$ for
  all $ v \in F_{\emptyset}$.

We now show one application of the free averaging algebras. An
averaging $R$-algebra $(A, f)$ is called {\bf primary} if it has no
proper averaging subalgebras. We will use the free averaging
$R$-algebra on the empty set, i.e. $(F_{\emptyset},
f_{\emptyset})$, to describe primary averaging $R$-algebras.
\begin{prop}
An averaging $R$-algebra $(A, f)$ is primary if and only if it is
isomorphic to a quotient averaging $R$-algebra
 $(F_{\emptyset}/I ,\bar f_{\emptyset})$, where $I$ is an
 averaging ideal of $(F_{\emptyset}, f_{\emptyset})$ such that
 $I \neq F_{\emptyset}$.
\end{prop}
\begin{proof}
 Recall that $F_{\emptyset} = R[y_{1_R}]$, and $f_{\emptyset}(v) = y_{1_R}v$ for
 all $ v \in F_{\emptyset}$.

 $( \Leftarrow )$  we prove that for any averaging ideal I of $(F_{\emptyset},
f_{\emptyset})$ such that  $I \neq F_{\emptyset}$, the quotient
averaging $R$-algebra
 $(F_{\emptyset}/I ,\bar f_{\emptyset})$ is primary.

 Suppose that $(S, \bar f_{\emptyset})|_S)$ is an averaging subalgebra of $(F_{\emptyset}/I ,
 \bar f_{\emptyset})$, then the identity element of $F_{\emptyset}/I$, i.e. $1_R + I  \in S$,
 and $y_{1_R}+ I = f_{\emptyset}(1_R)+ I = \bar f_{\emptyset}(1_R + I ) \in
 S$, hence $S = R[y_{1_R}]/I = F_{\emptyset}/I$.

$ ( \Rightarrow )$  According to the definition of free averaging
$R$-algebra, there exists a unique averaging homomorphism
\[ \hat{\varphi} : (F_{\emptyset}, f_{\emptyset}) \rightarrow (A,
f) .\] Since the image of $\hat {\varphi}$ is an averaging
subalgebra of $(A, f)$, and $(A, f)$ is primary, we conclude that
$\hat {\varphi}$ is surjective. Therefore if we take $I = ker \hat
{\varphi}$, we have $(A,f) \cong (F_{\emptyset}/I ,\bar
f_{\emptyset})$. \wbox
\end{proof}

\subsection{Free unitary averaging algebras and free\\
Reynolds-averaging algebras}

In this section, we consider  free objects for subcategories of
averaging $R$-algebras.

\begin{defn}
 Let $f$ be an averaging operator on  an $R$-algebra $A$. $f$ and
$(A, f)$ are  called {\bf unitary}  if
 $f(1_{A}) = 1_{A}$.
\end{defn}

\begin{defn} An averaging operator $f$  on  an $R$-algebra $A$ is
called a {\bf Reynolds-averaging operator} and $(A,f)$ is called a
{\bf Reynolds-averaging algebra }   if f also satisfies the {\bf
Reynolds identity }:
 \[ f(x)f(y) + f(f(x)f(y)) = f(xf(y)) + f(yf(x)) \]
   for all $x, y \in A$.
\end{defn}
Clearly, an averaging operator $f$ is a Reynolds operator if and
only it satisfies
\[f(f(x)f(y)) = f(x)f(y)\]
for all $x, y \in A$.

\begin{defn}
Let $X$ be a  set, $(F, f)$ a unitary averaging
(Reynolds-averaging) $R$-algebra, and $i: X \rightarrow F $ a map.
$(F, f)$ together with the map
  $i: X \rightarrow F $
  is said to be a {\bf free unitary averaging (free Reynolds-averaging) $R$-algebra on
  $X$ }
   if for any
 unitary  averaging (Reynolds-averaging) $R$-algebra $(B, g)$ and any map
 $\eta : X \rightarrow B$,  there exists a unique $R$-algebra
 homomorphism
  \[ \hat \varphi : (F, f) \rightarrow (B, g) \]
 such that $\hat \varphi \circ f = g \circ \hat \varphi  $ and $\hat \varphi \circ i = \eta $.
\end{defn}

Let $(F_{X}, f_{X})$ be the free averaging $R$-algebra on $X$, as
we constructed in the last subsection. Let $I_0$ be the ideal of
the $R$-algebra $F_X = R[X \cup Y]$ generated by the element
$y_{1_R} - 1$, i.e.
\[ I_0 = \{ \Sigma_i (y_{1_R} - 1_R)u_i : u_i \in F_X \}.  \]
Also let $I_1$ be the ideal of the $R$-algebra $F_X = R[X \cup Y]$
generated by the elements in the form $yy'y_{1_R} - yy'$, i.e.
\[ I_1 = \{ \Sigma_i (yy'y_{1_R} - yy')u_i :  y, y' \in Y , u_i \in F_X \}.  \]

 Actually, both $I_{0}$ and $I_{1}$ are averaging ideals of
$(F_{X}, f_{X})$. Therefore, $f_{X}$ induces  averaging operators
$\bar{f_{i}}$ on $F_{X}/I_{i}$  ( $ i = 0$ or $1$ ) by
 \[ \bar{f_{i}}( \alpha + I_{i}) = f_{X}(\alpha) + I_{i}. \]
\begin{prop}
The averaging operator $\bar{f_{0}}$ on  the quotient algebra
$F_{X}/I_{0}$ is unitary, and the averaging operator $\bar{f_{1}}$
on  the quotient $R$-algebra $F_{X}/I_{1}$ is a Reynolds-averaging
operator.
\end{prop}
\begin{proof}
Since $y_{1_R} - 1_R \in I_0$, we have $y_{1_R} + I_{0} = 1_{R} +
I_{0}$. Therefore
\begin{eqnarray*}
&&  \bar{f_{0}}( 1_{R} + I_{0}) = f_{X}(1_{R}) + I_{0}
\\ && = y_{1_R} + I_{0} = 1_{R} + I_{0}.
\end{eqnarray*}
Hence  $\bar{f_{0}}$ is unitary.
 To see $\bar{f_{1}}$ is a Reynolds-averaging operator, it is
 enough to notice that for all $\alpha = u_1v_1, \beta = u_2v_2 \in F_{X}$,
  where  $u_1, u_2 \in \Theta(X)$, $v_1, v_2 \in R[Y]$,  we
 have
 \begin{eqnarray*}
 && f_{X}(f_{X}(\alpha)f_{X}(\beta)) - f_{X}(\alpha)f_{X}(\beta)
 \\&& = f_X(y_{u_1}y_{u_2}v_1v_2) - y_{u_1}y_{u_2}v_1v_2
 \\&& =y_{1_R} y_{u_1}y_{u_2}v_1v_2 - y_{u_1}y_{u_2}v_1v_2
 \\ && =  (y_{u_1}y_{u_2}y_{1_R} - y_{u_1}y_{u_2})v_1v_2,
 \end{eqnarray*}
 therefore
 \[f_{X}(f_{X}(\alpha)f_{X}(\beta)) - f_{X}(\alpha)f_{X}(\beta) \in
 I_{1},\]
  and we have
  \[ \bar f_{1}(\bar f_{1}(\alpha + I_1)\bar f_{1}(\beta + I_1)) = \bar f_{1}(\alpha+ I_1) \bar f_{1}(\beta+I_1)
 .\]
\wbox
\end{proof}

Let $\pi_{i} : F_{X} \rightarrow F_{X}/I_{i}$ be the canonical
$R$-algebra homomorphisms , $i = 0$ or $1$ , and let
 $j: X \rightarrow F_{X}$ be the inclusion map.

\begin{prop}
(i) Together with the map
  $i_{0}= \pi_{0} \circ j : X \rightarrow F_{X}/I_{0}$,
   $(F_{X}/I_{0}, \bar{f_0})$ is a free unitary averaging
$R$-algebra on $X$.

(ii)  Together with the map
  $i_{1}= \pi_{1} \circ j : X \rightarrow F_{X}/I_{1}$,
     $(F_{X}/I_{1}, \bar{f_1})$ is a free
Reynolds-averaging $R$-algebra on $X$. \wbox \label{prop1}
\end{prop}
We omit the proof since it is similar to the proof of  Theorem
{\ref{freeOnX}}. However we would like to mention that

(i) $F_{X}/I_{0}$ is (isomorphic to) the polynomial ring
    $R[Z]$, where $Z = X \cup (Y - \{y_{1_R}\})$.

(ii)  Although  $F_{X}/I_{1}$ is no longer  a polynomial ring, its
element $u$ can be written uniquely in the form
\[  u = U + Vy_{0} + Wy_{0},  \]
 where $U  \in R[X \cup (Y - \{y_{0}\}]$, $V  \in R[X]$ and
 $W \in F_{X}$ and every monomial of $W$ contains exactly
 one element in $Y$. The multiplication can be easily described.

To conclude  this section, we will describe in general  the free
objects for the  subcategory of averaging $R$-algebras which are
defined by a set of equations.

Let $E$ be a set of some equations
\[ E = \{ \phi_{i}(t_{1}, t_{2},..., t_{m} , f_{E}) = 0: i \in I
\},  \] where $f_{E}$ is a function symbol of arity 1 and
  $t_{1}, ...,t_{m}$ are symbols which are not elements of
 $X \cup Y$. Let
 \[T = \{t_{1}, ..., t_{m} \}. \]

  For any $R$-algebra $A$, let
  \[M(T, A) = \{ maps \;\;\sigma: T \rightarrow A \}. \]
Each equation $\phi_{i}(t_{1}, ..., t_{m}, f_{E})=0$ induces an
ideal $I(\phi_{i})$ of $(F_{X}, f_{X})$ (the free averaging
$R$-algebra on the set $X$) which is generated by the elements of
the the following form
\[ \phi_{i}(\sigma(t_{1}),..., \sigma(t_{m}), f_X), \]
where $\sigma \in M(T, F_{X})$. Let
  $\displaystyle I_{E}=\sum_{i \in I}I(\phi_{i})$.
   We say that an averaging $R$-algebra
$(B, g)$ satisfies the set $E$, if for any $\sigma \in M(T,B)$ and
any $i \in I$
\[ \phi_{i}(\sigma(t_{1}),...,\sigma(t_{m}), g) = 0\]
holds in $(B, g)$. We give the following theorem without proof.
\begin{theorem}
The quotient $R$-algebra $F_{X}/I_{E}$ with the induced averaging
operator $\bar{f_{X}}$ given by $\bar{f_{X}}(u + I_{E}) = f_{X}(u)
+ I_{E}$ is an averaging $R$-algebra satisfying the equation set
$E$, and
 $(F_{X}/I_{E}, \bar{f_{X}})$ is a free object on $X$ for the subcategory of
 averaging $R$-algebras which satisfy the  set of equations  $E$.
\end{theorem}

\newpage
\section{Properties of free averaging algebras}
\subsection{Induced homomorphisms between free averaging algebras}
Let  $A$ and $B$ be  $R$-algebras, $(F_A,f_A)$ and $(F_B,f_B)$ be
the free averaging $R$-algebras on $A$ and $B$ respectively.  If
 $\theta : A \rightarrow B$ is an $R$-algebra homomorphism,   then
by the definition of the free averaging algebra on an algebra,
there exists   a unique averaging homomorphism
\[ \hat {\theta} : (F_{A}, f_{A}) \rightarrow (F_{B}, f_{B}),  \]
and the following diagram commutes:

\[\begin{array}{ccc}
(F_A,f_A) & \ola{\hat{\theta}} & (F_B,f_B)\\  \uap{i_A} &&
\uap{i_B}\\ A & \ola{\theta} & B.
\end{array} \]

We say that $\hat{\theta}$ is induced by $\theta$.

 The induced
homomorphism $\hat{\theta}$ can be described nicely in terms of
$\theta$, as the lemma below indicates.

\begin{lemma}
 For $a, a_{1}, \ldots , a_{k} \in A$, if
 $ \theta(a) = b, \theta(a_{i}) = b_{i}, 1 \leq i \leq k$,
then
\[\hat{\theta}(a \otimes (a_{1} \odot a_{2} \ldots \odot a_{k} ))= b \otimes (b_{1} \odot b_{2} \ldots \odot
b_{k})\] \label{lemtheta}
\end{lemma}
\begin{proof}
\begin{eqnarray*}
&& \hat{\theta}(a \otimes (a_{1} \odot a_{2} \ldots \odot a_{k} ))
\\ && = \hat{\theta}((a \otimes 1_{R}) f_{A}(a_{1} \otimes 1_{R})
\ldots f_{A}(a_{k} \otimes 1_{R}))
\\ && = \hat{\theta}(a \otimes 1_{R}) f_{B}(\hat{\theta}(a_{1}
\otimes 1_{R})) \ldots  f_{B}(\hat{\theta}(a_{k} \otimes 1_{R}))
\\ && = \hat{\theta} \circ i_{A}(a) f_{B}(\hat{\theta} \circ
i_{A}(a_{1})) \ldots f_{B}(\hat{\theta} \circ i_{A}(a_{k}))
\\ && = i_{B} \circ \theta(a) f_{B}(i_{B} \circ \theta
(a_{1})) \ldots f_{B}(i_{B} \circ \theta(a_{k}))
\\ && = (b \otimes 1_{R})(1_{R} \otimes b_{1}) \ldots (1_{R}
\otimes b_{k})
\\ && = b \otimes (b_{1} \odot b_{2} \ldots \odot b_{k}).
\end{eqnarray*}
So we have proved the lemma. \wbox
\end{proof}

Clearly,  if  $\theta$ is an $R$-algebra isomorphism, then
$\hat{\theta}$ is an averaging isomorphism of averaging algebras.
If
 $\eta : B \rightarrow C $ is another $R$-algebra homomorphism,
 $\hat {\eta} : (F_{B}, f_{B}) \rightarrow (F_{C}, f_{C})$ is the
 averaging homomorphism induced by $\eta$, then
 $\hat{\theta} \circ \hat{\eta} : (F_{A}, f_{A}) \rightarrow (F_{C},f_{C})$
 is the averaging homomorphism induced by  $\theta \circ \eta$.
 We also have the following theorem.

\begin{theorem}

With the notations above, the following are true .

(i) $\theta$ is injective if and only if $\hat{\theta}$ is
injective.

(ii)  $\theta$ is surjective if and only if $\hat{\theta}$ is
surjective.

(iii) If $A$ is a subalgebra of $B$, and $\theta$ is the inclusion
homomorphism, then $(F_{A}, f_{A})$ is an averaging subalgebra of
 $(F_{B}, f_{B})$,  $\hat{\theta}$ is the inclusion
 homomorphism, and if $A \neq B$ then $F_{A} \neq F_{B}$.
\label{thm4}
\end{theorem}
\begin{proof}
We have the following commutative diagram:

\[\begin{array}{ccccc}
(F_A,f_A) & \ola{\hat{\theta}_1} & (F_C,f_C) &\ola{\hat{\theta}_2}
&(F_B,f_B) \\ \uap{i_A} && \uap{i_C} && \uap{i_B}\\ A&
\ola{\theta_1} & C & \ola{\theta_2} &B
\end{array} \]
where $C = \theta(A)$, $\theta(a) = \theta_{1}(a)$ for all
  $a \in A$, $\theta_{2}$ is the inclusion map. Clearly
$\hat {\theta}_{1}$ is surjective, and $\hat {\theta}_{2}$ is the
inclusion map according to   Lemma {\ref{lemtheta}}. Since
$\theta_{2} \circ \theta_{1} = \theta $ and the diagram is
commutative, we have
 $\hat{\theta}_{2} \circ \hat{\theta}_{1} = \hat {\theta}$.

(i) if $\theta$ is injective, then $\theta_{1}$ is an isomorphism,
and $\hat{\theta}_{1}$ is an isomorphism, hence
 $\hat{\theta} = \hat{\theta}_{2} \circ \hat{\theta}_{1}$ is
 injective. Conversely, if $\hat{\theta}$ is injective, then
 $\hat{\theta}_{1}$ is injective, hence
 $i_{C} \circ \theta_{1} = \hat{\theta}_{1} \circ i_{A} $ is
 injective, therefore $\theta_{1}$ is injective, and finally
 $\theta = \theta_{2} \circ \theta_{1}$ is injective.

(ii)  if $\theta$ is surjective, then $C = B$, $F_{C} = F_{B}$,
 both $\theta_{2}$ and $\hat{\theta}_{2}$ are identity maps, hence
  $\hat{\theta} = \hat{\theta}_{2} \circ \hat{\theta}_{1}$ is
  surjective. Conversely, if $\hat{\theta}$ is surjective, then
 $\hat{\theta_{2}}$ is surjective, i.e.  $\hat{\theta_{2}}$ is the
 identity map, hence  $\theta_{2}$ is the identity map, and
  $\theta = \theta_{2} \circ \theta_{1}$ is surjective.

(iii)  It is clear that $(F_{A}, f_{A})$ is an averaging
subalgebra of
 $(F_{B}, f_{B})$,  $\hat{\theta}$ is the inclusion
 homomorphism  according to  Lemma {\ref{lemtheta}},  and
  if $A \neq B$ then $\theta$ is not surjective, hence
   $\hat{\theta}$ is not surjective, and
     $F_{A} \neq F_{B}$.  \wbox
\end{proof}
\begin{coro}
Let $\hat{\theta}: (F_{A},f_{A}) \rightarrow (F_{B},f_{B})$ be the
averaging homomorphism induced by the $R$-algebra homomorphism
$\theta : A \rightarrow B$. Then $\hat{\theta}$ is an averaging
isomorphism if and only if $\theta$ is an isomorphism. \wbox
\end{coro}

\subsection{Ascending chain conditions in  free averaging algebras}
In this subsection, we assume that $R$ is a noetherian ring. An
averaging $R$-algebra $(A,f)$ is said to be a {\em noetherian
averaging $R$-algebra } if there exist no infinite ascending chain
of averaging  ideals in $(A,f)$
\[ I_{1} \subset I_{2} \subset \ldots \subset I_{n} \subset \ldots
\]

 Note that  if $A$ is a noetherian
$R$-algebra, then $(A,f)$ is a noetherian averaging $R$-algebra,
since every averaging ideal of $(A,f)$ is an ideal of $A$.

\begin{theorem} The free averaging $R$-algebra $(F_{X}, f_{X})$ on
a set $X$ is a noetherian averaging $R$-algebra if and only if
  $ X = \emptyset $.
\label{thm5}
\end{theorem}
\begin{proof}
If $X = \emptyset$, then $Y = \{y_{0}\}$, and
 $F_{\emptyset} = R[y_{0}]$, which is a noetherian $R$-algebra,
 hence  $(F_{\emptyset}, f_{\emptyset})$ is a noetherian averaging
 $R$-algebra.
If $X \neq \emptyset$ then $Y$ is an infinite set.  There exists
an infinite ascending chain of ideals in $R[Y]$:
\[ I_{1} \subset I_{2} \subset \ldots \subset I_{n} \subset \ldots
\]
For each $i$, $ 0 <  i  < \infty$,  the averaging ideal of
$(F_{X}, f_{X})$ generated by $I_{i}$ is
\[ \bar{I}_{i} = \{ \sum_{k} \alpha_{k} \beta_{k} : \alpha_{k} \in
F_{X}, \beta_{k} \in I_{i} \}. \]
 These averaging ideals $\bar{I}_{i}$ of $(F_{X}, f_{X})$ are all
 distinct, since if we define an $R$-algebra homomorphism
 \[ \phi : F_{X} \rightarrow R[Y] \]
 \[ \phi(x) = 1, \;\;\; \phi(y) = y \]
 for all $x \in X$ and $ y \in Y$, then
 $\phi(\bar{I}_{i})=I_{i}$. Hence we have an infinite ascending
 chain of averaging ideals in $(F_{X}, f_{X})$
 \[ \bar{I}_{1} \subset \bar{I}_{2} \subset \ldots \subset \bar{I}_{n} \subset \ldots
\]
So we have proved the theorem. \wbox
\end{proof}

Recall that $(F_{X}, f_{X})$ is a free averaging $R$-algebra on
the $R$-algebra $R[X]$. When $X$ contains only one element, $R[X]$
is a noetherian $R$-algebra, but $(F_{X}, f_{X})$ is not a
noetherian averaging $R$-algebra. Therefore  $A$ is a noetherian
$R$-algebra does not implies that $(F_{A}, f_{A})$ is a noetherian
averaging $R$-algebra. But we do have the following.

\begin{theorem}
Let $A$ be an $R$-algebra, and $(F_{A}, f_{A})$ be the free
averaging $R$-algebra on $A$. If $(F_{A}, f_{A})$  is a noetherian
averaging $R$-algebra, then $A$ is a noetherian $R$-algebra.
\label{thm6}
\end{theorem}
\begin{proof}
If there exists an infinite ascending chain of ideals in $A$
\[ I_{1} \subset I_{2} \subset \ldots \subset I_{n} \subset \ldots
\]
Each $I_{i}$ is an averaging ideal of the averaging $R$-algebra
$(A, j_{A})$, where $j_{A}$ is the identity map of $A$. There
exists a unique averaging homomorphism
\[ \hat{\varphi} : (F_{A}, f_{A}) \rightarrow (A, j_{A}) \]
such that $j_{A} = \hat{\varphi} \circ i_{A}$. For each
 $i$, $ 0 <  i  < \infty$, let
   $\bar{I}_{i}=\hat{\varphi}^{-1}(I_{i})$, and $\bar{I}_{i}$ is an
   averaging ideal of $(F_{A}, f_{A})$. Clearly
   $\bar{I}_{i} \neq \bar{I}_{j}$ for $i \neq j$.   Hence we have an infinite
   ascending chain of averaging ideals in $(F_{A}, f_{A})$
\[ \bar{I}_{1} \subset \bar{I}_{2} \subset \ldots \subset \bar{I}_{n} \subset \ldots
\]
So we have proved the theorem.\wbox
\end{proof}

\subsection{Decision problems of averaging algebras}
We will give an application of our construction of free
averaging algebras on decision problems.
For notations and more information decision problems,
we refer the readers to {\cite{Bach}}.

We may wonder if a unitary averaging algebra is a
Reynolds algebra, or if a Reynolds operator $f$ on  $A$ must
satisfy
\[nf(xf(x)^{n-1}) = (n-1)f(f(x)^{n}) + f(x)^{n} .\]
This kind of problems can be formulated as the following decision
problems in equational logic: for two finite sets $E_{1}$ and
$E_{2}$ of equational axioms , decide if each equation of $E_{2}$
is a logical consequence of $E_{1}$. The most important well-known
method is to use the Knuth-Bendix completion procedure and
transform the set $E_1$  into a convergent rewriting system.
Unfortunately, this approach does not always work. In this
section, We will consider some decision problems related to
averaging algebras.  More precisely, we will consider the situations
where $E_{1}$ is one of the following:
\begin{eqnarray*}
&& (i)\ \  \{f(xf(y)) = f(x)f(y) \}
\\ && (ii)\ \   \{ f(xf(y)=f(x)f(y),\ f(x) = 1 \}
\\ && (iii)\ \  \{f(xf(y)) = f(x)f(y),\  f(f(x)f(y)) =f(x)f(y))
\}.
\end{eqnarray*}

We give the definitions first.
 Let $V$ be a countable set of variable symbols, $\bf {1}$ and $\bf {0}$ be
 two constant symbols, and $\bf {f}$ be a function symbol of arity 1.

 \begin{defn}
 $R$-algebraic terms are defined inductively as follows.

 (i) Every $v \in V$ is an $R$-algebraic term.

 (ii) both $\bf {1}$ and $\bf {0}$
 are  $R$-algebraic terms.

 (iii) If $t_{1}$ and $t_{2}$ are $R$-algebraic terms,
 then $(t_{1} + t_{2})$, $(t_{1} - t_{2})$, $(t_{1}t_{2})$ are
 $R$-algebraic terms.

 (iv) If $t$ is an $R$-algebraic term, $r \in R$, then
 $(rt)$ and ${\bf {f}}(t)$ are $R$-algebraic terms.

 (v) Only these are $R$-algebraic terms.

\end{defn}

\begin{defn}
An $R$-algebraic equation is a pair of $R$-algebraic terms
 $(t_{1}, t_{2})$, also  written  as    $t_{1} = t_{2}$.
\end{defn}

\begin{defn}
Let $f$ be an $R$-endomorphism of an $R$-algebra $A$. We say that
$f$ satisfies the  $R$-algebraic equation
\[ \Phi_{1}(v_{1}, v_{2}, \ldots , v_{k}, {\bf{1}}, {\bf{0}}, {\bf{f}}) =
 \Phi_{2}(v_{1}, v_{2}, \ldots , v_{k}, {\bf{1}}, {\bf{0}}, {\bf{f}}),\] where
 $\Phi_{i}(v_{1}, v_{2}, \ldots , v_{k}, \bf{1}, \bf{0}, \bf{f})$
 are
  $R$-algebraic terms that involves
 $v_{1}, v_{2},$ $ \ldots , v_{k} \in V$, the two constants symbols
 $\bf{1}$ and $\bf{0}$, and the function symbol $\bf{f}$, if for
 any $a_{1}, a_{2}, \ldots , a_{k} \in A$,
 \[ \Phi_{1}(a_{1}, a_{2}, \ldots , a_{k}, 1_{A}, 0_{A}, f) = \Phi_{2}(a_{1}, a_{2}, \ldots , a_{k}, 1_{A}, 0_{A}, f) \]
 holds in $A$.

 We say that $f$ satisfies a set of $R$-algebraic equations $E$, if
 $f$ satisfies each $R$-algebraic equation of $E$.
\end{defn}

\begin{defn}
Let $E_{1}$ and $E_{2}$ be two sets of $R$-algebraic equations. We
say that $E_{1}$ implies $E_{2}$ ,  if for any $R$-algebra $A$ and
any $R$- endomorphism $f$ on  $A$, whenever $f$ satisfies $E_{1}$ ,
it also satisfies $E_{2}$.
\end{defn}

We write $E_{1} \models E_{2}$ to indicate that $E_{1}$ implies
$E_{2}$.

Let
\begin{eqnarray*}
&& E_{a} = \{ {\bf{f}}(v_{1}{\bf{f}}(v_{2})) =
{\bf{f}}(v_{1}){\bf{f}}(v_{2}) \},
\\ && E_{ua} = \{ {\bf{f}}(v_{1}{\bf{f}}(v_{2})) =
{\bf{f}}(v_{1}){\bf{f}}(v_{2}), \;\; {\bf{f}}({\bf{1}}) = {\bf{1}}
\},
\\ && E_{ra} =\{ {\bf{f}}(v_{1}{\bf{f}}(v_{2})) =
{\bf{f}}(v_{1}){\bf{f}}(v_{2}), \;\;
{\bf{f}}({\bf{f}}(v_{1}){\bf{f}}(v_{2})) =
{\bf{f}}(v_{1}){\bf{f}}(v_{2}) \}.
\end{eqnarray*}

For an $R$-algebra $A$ and an $R$-endomorphism $f$ on  $A$,
  $(A, f)$ is an averaging $R$-algebra if and only if $f$
  satisfies $E_{a}$, it  is a unitary averaging $R$-algebra if and
  only if $f$ satisfies $E_{ua}$,  and it is a Reynolds-averaging
  $R$-algebra if and only if $f$ satisfies $E_{ra}$.

By taking the advantage of the explicit constructions of free
averaging algebras, we have the following theorem.

\begin{theorem}
If $E_{1}$, a set of $R$-algebraic equations,  is one of $E_{a}$,
$E_{ua}$ and $E_{ra}$, and $E_{2}$ is any finite set of
$R$-algebraic equations, then it is decidable whether
 $E_{1} \models E_{2}$  holds or not.
\label{thm7}
\end{theorem}
\begin{proof}
Without loss of generality, we assume that $E_{2}$ contains only
one equation in the form
\[ \Phi(v_{1}, v_{2}, \ldots , v_{k}, {\bf{1}}, {\bf{0}}, {\bf{f}})
= {\bf{0}}. \] We will only consider the case where $E_{1}
=E_{a}$, other cases can be treated similarly. Let
  $X = \{x_{1}, x_{2},\ldots, x_{k}\}$, and let  $(F_{X}, f_{X})$ be the free averaging
$R$-algebra on $X$.  $E_{1} \models E_{2}$ holds if and only if
any  averaging $R$-algebra $(A, f)$ satisfies
\begin{equation}
 \Phi(v_{1}, v_{2}, \ldots , v_{k}, {\bf{1}}, {\bf{0}}, {\bf{f}})
= {\bf{0}},
 \label{eq:e1}
\end{equation}
 which, we claim, is equivalent to
that
\begin{equation}
\Phi(x_{1}, x_{2}, \ldots , x_{k}, 1_{R}, 0_{R},  f_{X}) = 0_{R}
\label{eq:e2}
\end{equation}
 holds in $(F_{X}, f_{X})$.
  Equation ({\ref{eq:e2}}) must holds since  Equation
 ({\ref{eq:e1}})
  is supposed to  be satisfied
 by any averaging operator on  ant averaging algebra. Conversely, if Equation
 ({\ref{eq:e2}})  holds, then for any averaging $R$-algebra $(A, f)$
 and any $a_{1}, a_{2}, \ldots, a_{k} \in A$, define a map
 \begin{eqnarray*}
&& \eta : X \rightarrow A
\\ &&  \;\;\;\;\;   x_{i} \rightarrow a_{i}, \;\; 1 \leq  i \leq k.
 \end{eqnarray*}
 Then there exists a unique averaging $R$-algebra homomorphism
 \[ \hat{\eta} : (F_{X}, f_{X}) \rightarrow (A, f) \]
 such that $\hat{\eta} \circ i_{X} = \eta$.
 We have
\begin{eqnarray*}
&& \Phi(a_{1}, a_{2}, \ldots , a_{k}, 1_{A}, 0_{A},  f )
  \\ && =\hat{\eta}(\Phi(x_{1}, x_{2}, \ldots , x_{k}, 1_{R}, 0_{R},  f_{X}))
  \\ && = \hat{\eta}(0_{R}) = 0_{R}.
 \end{eqnarray*}
So in order to prove (or disprove)  that $E_{1} \models E_{2}$
holds, we only need to prove (or disprove) that Equation
({\ref{eq:e2}})  holds, which can be checked routinely, since
$F_{X}$ is a polynomial algebra. \wbox
\end{proof}
\begin{ex}
We claim that every unitary averaging operator is an Reynolds
operator, i.e. $E_1 = E_{ua}$  implies
  $E_2 = \{\bf{f}(\bf{f}(v_{1}){\bf{f}}(v_{2}))=
\bf{f}(v_{1})\bf{f}(v_{2}) \}$.
\end{ex}
\begin{proof}
Let $X = \{x_1, x_2 \}$, $Y = \{y_{\theta}: \theta \in \Theta(X)
\}$, and $Z = X \cup (Y - \{y_{1_R} \})$. Then $(R[Z], f_0)$ is a
free unitary averaging algebra on $X$, where $f_0$ is a unitary
averaging operator defined by

(i) $f_0(1_R) = 1_R$, and

 (ii) for each $1_R \neq u \in \Theta(X)$, $v \in
R[Y-\{y_{1_R}\}]$, $f_0(uv) = y_uv$, and

(iii) for each $v \in R[Y-\{y_{1_R}\}]$, $f_0(v) = v$.

We only need to check that $f_0(f_0(x_1)f_0(x_2)) =
f_0(x_1)f_0(x_2)$. Since $f_0(x_1)f_0(x_2) = y_{x_1}y_{x_2} \in
R[Y - \{y_{1_R}\}]$, we have
\[ f_0(f_0(x_1)f_0(x_2)) = f_0(y_{x_1}y_{x_2}) = y_{x_1}y_{x_2}
=f_0(x_1)f_0(x_2).\] Hence the claim is true, i.e., every unitary
averaging operator is a Reynolds operator.  \wbox
\end{proof}

\begin{ex}
We claim that an averaging operator is not necessarily a Reynolds
operator, i.e. $E_a$ does not implies $E_{ra}$.
\end{ex}
\begin{proof}
Let $(F_X, f_X)$ be the free averaging $R$-algebra on $X = \{x_1,
x_2 \}$. We only need show that
\[ f_X(f_X(x_1)f_X(x_2)) \neq f_X(x_1)f_X(x_2) .\]
Since $f_X(f_X(x_1)f_X(x_2))= y_{1_R}y_{x_1}y_{x_2}$,
  $f_X(x_1)f_X(x_2) = y_{x_1}y_{x_2}$, we know that $(F_X, f_X)$
  does not satisfies $E_{ra}$. Hence the claim is true. \wbox
\end{proof}

\newpage

\section{Lie algebraic structures induced by averaging operators}

\subsection{ Induced Lie bracket operation }

It is well-known that for an algebra $A$ (not necessarily
commutative), we can define a Lie bracket operation on $A$:
\[ [x,y] = xy - yx
.\]
 This operation becomes trivial when $A$ is commutative. The
operation can be rewritten as
 \[ [x,y] = x i_{A}(y) - y i_{A}(x), \]
where $i_{A}$ is the identity map of $A$, which is an averaging
operator on  $A$. This approach can be generalized and  make sense
even in the case where $A$ is commutative by using another
averaging operator on  $A$ instead of $i_{A}$.

\begin{defn}

Let $A$ be an algebra over ring $R$, and $f \in \Avg(A)$. Define
\begin{equation}
[x,y]_{f} = xf(y) - yf(x),\ x,\ y\in A. \label{eq:bax1}
\end{equation}

\end{defn}

\begin{theorem}
When $f$ is an averaging operator on  $A$,
the $R$-module $A$
becomes a Lie algebra under the bracket operation $[  ,  ]_{f}$.
\label{thm8}
\end{theorem}

\begin{proof} Clearly, the bracket operation is bilinear and $[x,x]_{f}
=0$, for all $ x \in A$. We can also verify that
\begin{eqnarray*}
&& [[x,y]_{f}, z]_{f} + [[y,z]_{f}, x]_{f} + [[z,x]_{f}, y]_{f}\\
&& = [xf(y)-yf(x), z]_{f} + [yf(z)-zf(y), x]_{f} + [zf(x)-xf(z),
y]_{f} \\ && = (xf(y)-yf(x))f(z) - zf(xf(y)-yf(x))  \\ && +
(yf(z)-zf(y))f(x) - xf(yf(z)-zf(y))\\ && + (zf(x)-xf(z))f(y) -
yf(zf(x)-xf(z)) \\
 && = 0,
\end{eqnarray*}
that is, $[,]_{f}$ is a Lie bracket operation.\wbox
\end{proof}

We will denote the Lie algebra $(A, [ , ]_{f})$ by $A_{f}$.

\subsection{Conditions under which  a Lie bracket operation
is induced by an averaging operator }

We point out that if we remove the requirement that the R-module
endomorphism $f$ is an averaging operator, $[ , ]_{f}$ may still
be a Lie bracket operation on $A$.

\begin{ex}
Let $A$ be the set of  the complex numbers , $R$  the set of real
numbers. $A$ is a 2-dimensional algebra over $R$. if $z = x + yi
$,  x, y are real numbers,  define
\[ g(z) = yi,  \]
then $g$ is an R-module endomorphism on  $A$. Note that $g$ is not
an averaging operator, but $[ , ]_{g}$ is a Lie bracket operation.
If we define \[ f(z) = g(z) - z = -x  \] then $f$ is an averaging
operator and $ [ ,]_{g} = [ , ]_{f} $.
\end{ex}

\noindent We would like to know under what conditions, a Lie
bracket operation is induced by a $R$-module endomorphism, or even
an averaging operator. We also want to know when two $R$-module
endomorphisms induce the same Lie bracket operation.

\begin{prop}
A Lie bracket operation $[,]$ on an $R$-algebra $A$ is induced by
an $R$-module endomorphism $f$ if and only if
\begin{displaymath} [x,y] = x[1_{A}, y] + y[x, 1_{A}]
\end{displaymath}for all $ x, y \in A$.
\label{prop2}
\end{prop}
\begin{proof} $( \Rightarrow )$ Suppose that the Lie bracket
operation is induced by an $R$-module endomorphism $f$ on  $A$.
Then
\begin{eqnarray*}
&& [x,y] = [x,y]_{f}\\ && = xf(y)-yf(x)
\\ && = xf(y)-xyf(1_{A}) + xyf(1_{A})-yf(x)  \\ && = x[1_{A}, y]_{f} + y[x,
1_{A}]_{f} \\ && = x[1_{A}, y] + y[x, 1_{A}].
\end{eqnarray*}
$( \Leftarrow )$ Define $f(x) = [1_{A}, x]$,  then $f$ is an
$R$-module endomorphism, and
\begin{eqnarray*}
&& [x,y]_{f} = xf(y)-yf(x) = x[1_{A},y] - y[1_{A},x]   \\ && =
x[1_{A},y] + y[x, 1_{A}] = [x, y].
\end{eqnarray*}
\end{proof}
Note in the second part of the proof, we could also take
 $ f(x) = [1_{A}, x] + xt $, where $ t \in  A$ can be  arbitrarily
chosen. The following proposition gives a general result
concerning this phenomenon.
\begin{prop}
Let $f_{1}$ and $f_{2}$ are $R$-module endomorphisms on  $A$.
$[,]_{f_{1}} =[,]_{f_{2}} $ if and only if  $f_{1}(x) = f_{2}(x) +
xt $ for some $t \in A$.
\end{prop}
\begin{proof}
We claim that an $R$-module endomorphism $f$ on  $A$ satisfies
\[ xf(y) - yf(x) = 0 \]
if and only if $f(x) = xt $ for some $ t \in A$.

If the equation holds, let $y = 1_{A}$, and we have $f(x) = xt$,
where $t=f(1_{A})$. Conversely if $f(x) = xt$ for some $t \in A$,
then
\[ xf(y) - yf(x) =xyt - yxt = 0. \]

The theorem follows immediately from the fact that
 $[,]_{f_{1}}=[,]_{f_{2}} $ is equivalent to
 \[   xf(y)- yf(x) = 0,  \] where  $f = f_{1} - f_{2}$.
\end{proof}

\begin{coro}
If a Lie bracket operation $[,]$ on an $R$-algebra $A$  satisfies
\begin{displaymath} [x,y] = x[1_{A}, y] + y[x, 1_{A}]
\end{displaymath}for all $ x, y \in A$, then it is induced by any of  the
$R$-module endomorphisms on  $A$ of the following form
          \[ f_{t}(x) = [ 1, x] + xt, \]
where $t \in A$ is a fixed element,  and it is not induced by any
other  $R$-module endomorphisms on  $A$. \wbox
 \label{cor1}
\end{coro}

Although it is not necessary for an $R$-module endomorphism $f$ to
be an averaging operator on $A$ in order to define a Lie bracket
operation, we are only interested in averaging operators when
discussing  the induced Lie algebraic structures in this thesis.
This restriction grants much richer results as we will see in the
next section. The following theorem describes the Lie bracket
operations on $A$ which can be induced by averaging operators.

For a Lie bracket operation $L = [ , ] $ on an $R$-algebra $A$,
define
\[ \Gamma(L) = \{ a \in A : [ax,y] = a[x,y], x,y\in A \}, \]
and for each $t \in A$ define an $R$-module endomorphism $L_{t}$
by
\[ L_{t}(x) = [1,x] + xt,   \]
$\Gamma(L)$ is a subalgebra of $A$.
\begin{theorem}
A Lie bracket operation $L = [ , ]$ on an $R$-algebra $A$ is
induced by an averaging operator $f$ on $A$ if and only if
$[x,y]=x[1_{A}, y]$ $+y[x, 1_{A}] , x, y \in A$ and for some fixed
$t \in A$ , $L_{t}(A) \subseteq \Gamma(L)$.
 \label{thm9}
\end{theorem}

\begin{proof} $( \Rightarrow )$ If $[x,y] = [x,y]_{f}$ then
$[x,y] = x[1_{A},y] + y[x, 1_{A}] $ by Proposition {\ref{prop2}}.
Furthermore, by taking $t = f(1_{A})$ we have $L_{t} = f$, and for
any $a, x, y \in A$
\begin{eqnarray*}
&& [ L_{t}(a)x, y] = [f(a)x, y] \\ && = f(a)xf(y) - yf(f(a)x) \\
&& = f(a)xf(y) - yf(x)f(a) \\ && = f(a)[x,y] = L_{t}(a)[x,y].
\end{eqnarray*}
$( \Leftarrow )$ Let $f = L_{t}$. We know that $[,]$ is induced by
$L_t$ according to Corollary {\ref {cor1}}.  Therefore the Lie
bracket operation is induced by $f$.  Furthermore
\begin{eqnarray*}
  && f(xf(y)) = [1_{A}, xf(y)] + xf(y)t
\\ && = [1_{A}, x]f(y) + xf(y)t
\\ && = ([1_{A},x] + xt )f(y)
\\ && = f(x)f(y).
\end{eqnarray*}
We conclude that $f$ is an averaging operator on $A$.\wbox
\end{proof}

\begin{prop}
If $R$ is a field, and an $R$-algebra $A$ is a 2-dimensional
vector space over $R$. Any Lie bracket operation on $A$ is induced
by some averaging operator on  $A$. \label{prop2dim}
\end{prop}
\begin{proof}
Let $\{ 1_{A}, \varepsilon \}$ be a basis of $A$,   $L = [,]$ be a
Lie bracket operation on $A$, and
 $[1_{A}, \varepsilon] = r_{1}1_{A} + r_{2}\varepsilon$. Define
        \[\theta_{L}( x, y ) = x[1_{A}, y] + y[x, 1_{A}] - [x,y] , \;\; x, \; y
        \in A.  \]
$\theta_{L}$ is a $R$-bilinear map, $\theta_{L}(x,x) =0$ for all
$x \in A$. $L = [,]$ is induced by an $R$-module endomorphism on
$A$ if and only if $\theta_{L}(x, y)=0$ for all $x, y \in A$.
Since
    \[\theta_{L}(1_{A}, 1_{A}) =  \theta_{L}( \varepsilon, \varepsilon
   )=0,  \]
\[\theta_{L}(1_{A}, \varepsilon) = \theta_{L}(\varepsilon, 1_{A})
=0, \]   we have $\theta_L (x,y) = 0 $ for all $x,y \in A$. Hence
$L = [,]$ is induced by
       \[ f(x) = [1_A, x ] + tx,  \]
where $t \in A$ can be arbitrarily chosen. If we take
 $t = -r_{2}1_{A}$, then for any $x, y \in A$,
     \[ x = m_{1}1_{A} + m_{2}\varepsilon, \]
     \[ y = n_{1}1_{A} + n_{2}\varepsilon, \]
where $ m_{1}, m_{2},n_{1}, n_{2} \in R$, we have
\[ f(x) = [1_A, x] -r_{2}x  = (m_{2} r_{1}- r_{2}m_{1})1_{A},\]
\[ f(y)= [1_A, y] -r_{2}y  = (n_{2} r_{1}- r_{2}n_{1})1_{A}.\]
Hence we have
\[ f(xf(y)) = f(x)f(y),  \]
that is, $f$ is an averaging operator on  $A$.\wbox
\end{proof}

The result of Proposition {\ref{prop2dim}} is not generally true
for $R$-algebras of higher dimensions.

\begin{ex}
Let $R$ be the set of real numbers, $A = R^{3}$. $A$ is a
3-dimensional vector space over $R$, and
   \[{\varepsilon}_{1} = (1, 0, 0), \]
   \[{\varepsilon}_{2} = (0, 1, 0), \]
   \[{\varepsilon}_{3} = (0, 0, 1), \]
is a basis of $A$. If we define a multiplication on $A$
\[ (x_{1}, x_{2}, x_{3})(y_{1}, y_{2}, y_{3}) = (x_{1}y_{1},
x_{2}y_{2}, x_{3}y_{3}), \] then $A$ is a commutative $R$-algebra.
A Lie bracket operation can be defined on $A$ as follows:
    \[[\varepsilon_{1}, \varepsilon_{2}] = \varepsilon_{1}, \]
    \[[\varepsilon_{1}, \varepsilon_{3}] = \varepsilon_{2}, \]
    \[[\varepsilon_{2}, \varepsilon_{3}] = \varepsilon_{3}. \]
If this Lie bracket operation was induced by some averaging
operator $f$ on  $A$, then
\[ f(\varepsilon_{1}) = f([\varepsilon_{1}, \varepsilon_{2}]) = 0,
\]
\[ f(\varepsilon_{2}) = f([\varepsilon_{1}, \varepsilon_{3}]) = 0,
\]
\[ f(\varepsilon_{3}) = f([\varepsilon_{2}, \varepsilon_{3}]) = 0.
\]
Hence $f = 0$. But zero map  is impossible to induce any
nontrivial Lie bracket operators.  A contradiction.
\end{ex}

\subsection{Finite dimensional algebras over a field}
 Now we consider the condition under which a Lie bracket
operation on a finitely dimensional algebra over a field  is
induced by an averaging operator. Let $A$ be a $n$-dimensional
$R$-algebra $(n > 0) $, $R$ is a field, and $L = [,]$ be a Lie
bracket operation on  $A$. Recall we define
\[ \Gamma(L) = \{ a \in A : [ax,y] = a[x,y], x,y\in A \}. \]
$\Gamma(L)$ is a subalgebra of $A$ and
 $\{r1_{A}: r \in R \} \subseteq \Gamma(L) $.
Suppose that $\{ \varepsilon_{1}, \ldots, \varepsilon_{p} \}$ is a
basis of $\Gamma(L)$,
  $\{\varepsilon_{1}, \ldots, \varepsilon_{p}, \delta_{1}, \ldots,\delta_{q}\}$
   is a basis of $A$, $ p, q > 0, p+q = n$, and for each $i$ and
   $j$, $ 1 \leq i \leq p$, $ 1 \leq j \leq q$,

\[ \displaystyle \varepsilon_{i} \delta_{j} = \sum_{1 \leq k \leq
p}a_{k}^{(i, j)}\varepsilon_{k} + \sum_{1 \leq k \leq q}b_{k}^{(i,
j)}\delta_{k},\]
\[ \displaystyle [1_{A}, \delta_{j}] = \sum_{1 \leq k \leq
p}c_{k}^{(j)}\varepsilon_{k} - \sum_{1 \leq k \leq q}d_{k}^{(
j)}\delta_{k},\] let
\[ B_{j} = \left[ \begin{array}{cccc}
b^{(1,j)}_{1} & b^{(2,j)}_{1} & \ldots & b^{(p,j)}_{1}
\\ b^{(1,j)}_{2} & b^{(2,j)}_{2} & \ldots & b^{(p,j)}_{2}
\\ \ldots & \ldots & \ldots & \ldots
\\ b^{(1,j)}_{q} & b^{(2,j)}_{q} & \dots & b^{(p,j)}_{q}
\end{array} \right] \;\;\;\;\;
and \;\;\;\;  \beta_{j} = \left[ \begin{array}{c} d^{(j)}_{1} \\
d^{(j)}_{2} \\ \vdots \\ d^{(j)}_{q} \end{array} \right ] \] then
we have the following theorem.
\begin{theorem}
With the notations defined above, the Lie bracket operation
 $L =[,]$ of $A$ is induced by an averaging operator on  $A$ if and only if
 \[ [x, y] = x[1_{A}, y] + y[x, 1_{A}] \;\; x, y \;\; \in A, \] and  the
following system of linear equations
\[  \left[ \begin{array}{c}
B_{1} \\ B_{2} \\ \vdots \\ B_{q} \end{array}\right ] \left[
\begin{array}{c} x_{1} \\ x_{2} \\ \vdots \\ x_{p}  \end{array}\right ] =
\left[
\begin{array}{c} \beta_{1} \\ \beta_{2} \\ \vdots \\ \beta_{q}
\end{array} \right ] \]
has a solution in $R^{p}$.
 \label{thm11}
\end{theorem}
\begin{proof}
$L = [,]$ is induced by an $R$-module endomorphism on  $A$  if and
only if
 \[ [x, y] = x[1_{A}, y] + y[x, 1_{A}] \;\; x, y \;\; \in A. \]
If this is true, then $L = [,]$ is induced by an averaging
operator if and only if $L_{t}(A) \subseteq \Gamma(L)$ for some
$R$-module endomorphism $L_{t} : A \rightarrow A$ defined as
\[ L_{t}(x) = [1_{A}, x] + xt, \]
where $t \in A $ is a fixed element. Such element $t$, if exists,
must be an element of $\Gamma(L)$, since $ t = L_{t}(1_{A})$.
Hence $L=[,]$ is induced by an averaging operator is equivalent to
that there exist $x_{1}, \ldots , x_{p} \in R$ such that for
 $\displaystyle t = \sum_{1 \leq i \leq p} x_{i} \varepsilon_{i} $,
 \[ L_{t}(\delta_{j}) = [1_{A}, \delta_{j}] + \delta_{j} t  \in
 \Gamma(L) \]
 holds for all $1 \leq j \leq q$.
Since for each $j$ ( $1 \leq j \leq q $),
\begin{eqnarray*}
 && L_{t}(\delta_{j}) = [1_{A}, \delta_{j}] + \delta_{j}t
 \\ &&  =
\sum_{1 \leq k \leq p}c_{k}^{(j)}\varepsilon_{k} - \sum_{1 \leq k
\leq q}d_{k}^{( j)}\delta_{k} + \sum_{1 \leq i \leq p}
x_{i}(\varepsilon_{i} \delta_{j})
\\ && = \sum_{1 \leq k \leq p}c_{k}^{(j)}\varepsilon_{k} - \sum_{1 \leq k
\leq q}d_{k}^{( j)}\delta_{k}
\\ &&\ \ \ +  \sum_{1 \leq i \leq q}
x_{i}(\sum_{1 \leq k \leq p}a_{k}^{(i, j)}\varepsilon_{k}) +
\sum_{1 \leq i \leq q} x_{i}(\sum_{1 \leq k \leq q}b_{k}^{(i,
j)}\delta_{k}),
\end{eqnarray*}
hence for $L_{t}(\delta_{j})$ to belong to $\Gamma(L)$, the
following must be true:
\[ \sum_{1 \leq i \leq p}x_{i}(\sum_{1 \leq k \leq q}b_{k}^{(i,
j)}\delta_{k}) =  \sum_{1 \leq k \leq q}d_{k}^{( j)}\delta_{k}, \]
that is
\[ \sum_{1 \leq k \leq q}(\sum_{1 \leq i \leq p}b_{k}^{(i,
j)}x_{i})\delta_{k} =  \sum_{1 \leq k \leq q}d_{k}^{(
j)}\delta_{k},
\]
and the conclusion of the theorem follows from this discussion
immediately. \wbox
\end{proof}

\newpage
\section{Properties of the induced Lie algebraic structures}
\subsection{Subalgebras, ideals, and homomorphisms }
Suppose that $f$ is an averaging operator on  an $R$-algebra $A$,
an ideal of $A$ is said to be an ideal of $(A,f)$ if it is
invariant under $f$, and a subalgebra of $A$ is said to be an
subalgebra of $(A,f)$ if it is invariant under $f$. An ideal  of a
Lie algebra $L$ is a submodule $I$ of $L$ such that $[x, I]
\subseteq I $ for all $x \in L$, and a subalgebra of $L$ is a
submodule $L_{1}$ of $L$ such that $[L_{1}, L_{1}] \subseteq
L_{1}$.  Also recall that we use $A_{f}$ to denote the induced Lie
algebra $(A, [,]_{f})$ by an averaging operator $f$ on  $A$.
\begin{prop}
Let $f$ be an averaging operator on  $A$,  $A_{1}$ an averaging
subalgebra of $(A,f)$, and $I$ an averaging ideal of $(A,f)$. then

(i) $A_{1}$ is a subalgebra of the Lie algebra $A_{f}$.

(ii) $I$ is an ideal of the Lie algebra $A_{f}$.

(iii) the quotient Lie algebra $A_{f}/I$ is induced by the
averaging operator
\[ \overline{f}(a+I) = f(a) + I \]
of $A/I$, i.e.
\[ A_{f}/I = (A/I, \overline{f}). \]
\end{prop}
\begin{proof}

(i) For $a_{1}, a_{2} \in A_{1}$, we have
  $f(a_{1}), f(a_{2}) \in A_{1} $, hence
   \[ [ a_{1}, a_{2}]_{f} = a_{1}f(a_{2}) - a_{2}f(a_{1}) \in
   A_{1} .\]

(ii) For $a \in I$, $f(a) \in I$. Hence for any $x \in A$, we have
    \[ [a, x]_{f} = af(x) - xf(a) \in I. \]

(iii) We use $[,]$ to denote the Lie bracket operation in
$A_{f}/I$.  For $a_{1}, a_{2} \in A_{1}$,
\begin{eqnarray*}
&& [a_{1}+I,a_{2}+I ] = [a_{1},a_{2}]_{f} + I
\\ && = a_{1}f(a_{2}) - a_{2}f(a_{1}) + I
\\ && = (a_{1}+I)( f(a_{2}) + I) - (a_{2} + I)(f(a_{1}) + I)
\\ && = (a_{1}+I)( \overline f(a_{2} + I)) - (a_{2} + I)(\overline f(a_{1} + I))
\\ && = [a_{1}+I,a_{2}+I ]_{\overline f}.
\end{eqnarray*}
So we have proved the proposition. \wbox
\end{proof}
\\[5 pt]

\begin{prop}
An averaging homomorphism $\varphi: (A, f) \rightarrow (B,g)$ is
also a Lie algebra homomorphism, i.e. (use the notations mentioned
above)
\[ \varphi([x,y]_{f}) = [\varphi(x), \varphi(y)]_{g} \]
\end{prop}
\begin{proof}
For any $x, y \in A$, we have
\begin{eqnarray*}
&& \varphi([x,y]_{f}) = \varphi( xf(y) - yf(x))
\\ &&= \varphi(x) \varphi(f(y)) - \varphi(y) \varphi(f(x))
\\ &&= \varphi(x) g(\varphi(y)) - \varphi(y) g(\varphi(x))
\\ && = [\varphi(x), \varphi(y)]_{g}.
\end{eqnarray*}
Therefore, $\varphi$ is also a Lie algebra homomorphism. \wbox
\end{proof}

\subsection{Solvable and nilpotent induced Lie algebras}
For a Lie algebra $L$, define
\begin{eqnarray*}
&& L^{(0)} = L, \\ && L^{(k)}= [L^{(k-1)},L^{(k-1})], \; for\;  k
> 0.
\end{eqnarray*}
 $L^{(k)}$, $k \geq 0$ are ideals, and
   $L^{(k)} \subseteq L^{(k-1)}$,  $k >0$.
   $L$ is called {\em solvable} if $L^{(k)}=0$ for
some $k >0 $.

Another sequence of ideals are  defined as
\begin{eqnarray*}
&& L^{1} = L, \\ && L^{k}= [L^{k-1},L], \; for \;  k > 1.
\end{eqnarray*}
$L^{k}$, $k > 0$ are ideals, and $L^{k} \subseteq L^{k-1} $,
 $k > 1$. $L$ is said to be {\em nilpotent} if $L^{k} = 0$ for some $k > 1$.
 The {\em nilpotent radical } of $L$ is the unique maximal
 nilpotent ideal ( if any ) of $L$, denoted by $Nr(L)$. $L$ is
 nilpotent if and only if $Nr(L) = L$.

\begin{prop}
Let $f$ be an averaging operator on  $A$. Then (i) $A_{f}$ is
solvable of length 2, i.e. $A_{f}^{(2)}=0$. (ii)  $A_{f}$ is
nilpotent if and only if
 \[ f(A)^{k} \subseteq \{ a \in A : aA_{f}^{2}=0 \}\]for some $k >0$.
\label{prop3}
\end{prop}
\begin{proof}
(i) For any $a_{1}, a_{2} \in A$, \[ f([a_{1}, a_{2}]_{f}) =
f(a_{1}f(a_{2}) - a_{2}f(a_{1})) = 0, \] therefore \[
[[x,y]_{f},[w,v]_{f}]_{f} = [x,y]_{f}f([w,v]_{f}) -
[w,v]_{f}f([x,y]_{f}) =0 \] for all $ x,y,w,v \in A$, i.e.
$A_{f}^{(2)}=0$. (ii) Computation shows that for $k > 2$ and
 $x_{1}, x_{2},..., x_{k} \in A$
\[ \underbrace{[...[[}_{k-1}x_{1}, x_{2}]_{f},x_{3}]_{f}, ..., x_{k}]_{f} =
[x_{1}, x_{2}]_{f}f(x_{3})...f(x_{k}) \] holds, hence
  $A_{f}^{k} =0$ if and only if
$ f(A)^{k-2} \subseteq \{ a \in A : aA_{f}^{2}=0 \}$.
\wbox
\end{proof}

\begin{prop}
Let $A$ be a domain and $f$ be an averaging operator on $A$.

 (i) If $ker(f) =0$, then $Nr(A_{f}) =A_{f}$.

 (ii) If $ker(f) \neq 0$, then $Nr(A_{f}) = ker(f)$.
\label{NrProp}
\end{prop}
\begin{proof}
If $ker(f)=0$, then for all $x, y \in A$, $[x, y]_{f}=0$,  since
\[ f([x,y]_{f}) = f(x)f(y) - f(y)f(x) = 0. \]
Therefore $A_{f}^{2} = 0$ and $Nr(A_{f}) = A$.

If $ker(f) \neq 0$, since clearly it is a nilpotent ideal of
$A_{f}$, we have $ker(f)\subseteq Nr(A_{f})$.
 Now take $0 \neq a \in Nr(A_{f})$. There exists an integer $n > 0$ such that
 $(ad_{f}(a))^{n} = 0$, $(ad_{f}(a))^{n-1} \neq 0$.  \cite{Jac}

If $n = 1$, take $0 \neq x \in ker(f)$,
\begin{eqnarray*}
&& 0=ad_{f}(a)(x) = [a,x ]_{f}
\\ && = af(x) - xf(a)
\\ && = -xf(a),
\end{eqnarray*}
therefore $f(a) = 0$ and $a \in ker(f)$.

If $n > 1$, take  $0 \neq x \in A$ such that
 $(ad_{f}(a))^{n-1}(x) \neq 0$. We have
\begin{eqnarray*}
&& 0 = (ad_{f}(a))^{n}(x)
\\&& = [a, (ad_{f}(a))^{n-1}(x)]_{f}
\\ && = -(ad_{f}(a))^{n-1}(x)f(a),
\end{eqnarray*}
therefore $f(a) = 0$ and $a \in ker(f)$. \wbox
\end{proof}

\begin{prop}
Let $A$ be an $R$-algebra without zero divisor, and $f \in
\Avg(A)$. The following are equivalent.

(i) $A_{f}$ is nilpotent;

(ii) $ker(f) = 0$ or $ker(f) = A$.

(iii)$f(x) = xt, x \in A$ for a fixed $t \in A$.

(iv) $A_{f}^{2} = 0$. \label{NrProp2}
\end{prop}

\begin{proof}
$(i) {\Rightarrow} (ii)$. If $A_{f}$ is nilpotent and $ker(f) \neq
0$ then $A = Nr(A_{f}) = ker(f)$ by proposition {\ref {NrProp}}.

$(ii) {\Rightarrow} (iii)$. If $ker(f) = A$, $f(x) =0 =  x0$, for
all $x \in A$. If $ker(f) = 0$, take $t = f(1)$, and
 \[f([1, x]_{f}) = f(f(x)- x(f(1)) = f(1)f(x)- f(x)f(1) = 0 \]
implies that
\[ 0 = [1, x]_{f} = f(x) - xf(1) = f(x) - xt. \]

$(iii) {\Rightarrow} (vi)$. For all $x, y \in A$,
\[ [x , y]_{f} = xf(y) - yf(x) = xyt - yxt = 0. \]
therefore $A_{f}^{2} = 0$.

$(vi) {\Rightarrow} (i)$. Trivial. So we have proved the
proposition. \wbox
\end{proof}

\subsection{Eigenvalues and eigenvectors of $ad_{f}(a)$ }
>From now on, we require that the ring $R$ is a field. So the $R$
-algebra $A$ is a vector space, and  we can consider the
eigenvalues, eigenvectors, and matrices of relevant $R$-linear
maps, etc. For a $R$-linear map $f$ and its eigenvalue $k \in R$,
we use the notation
\[ V_{k}^{f} = \{ x \in A: f(x) = kx \} \]
to denote the eigenspace of  $k$.

Let $f$ be an averaging operator on  $A$. For each $a \in A$, let
$ad_{f}(a)$ be the $R$-module endomorphism on  $A$ such that
$ad_{f}(a)(x)= [a, x]_{f}, x \in A$. $ad_{f}(a)$ is a derivation
on  $A_{f}$, and  $ad_{f}: A_{f} \rightarrow End_{R}(A)$ is a Lie
algebra homomorphism. {\cite {Jac}} {\cite {Kna}}

If $R$ is a field, $A$ is an $R$-algebra, and $f \in \Avg(A)$,
then for each $a \in A$, $0$ is an eigenvalue of $ad_{f}(a)$: if
$a = 0$, $ad_{f}(a)$ is the zero map; if $a \neq 0$, then
$ad_{f}(a)(a) = 0 = 0a$. So we only need to consider nonzero
eigenvalues for $ad_{f}(a)$.

\begin{prop}
Let $R$ be a field, $A$ be an $R$-algebra without zero divisor,
and $f \in \Avg(A)$. For each $0 \neq a \in A$ :

(i) if $ker(f) = 0$,  then $0$ is the unique eigenvalue of
 $ad_{f}(a)$,  and   $V_{0}^{ad_{f}(a)} = A$.

(ii) if $ker(f) \neq 0$, and $f(a) \neq k1_{A}$  for any
 $0 \neq k \in R$, then  $ad_{f}(a)$  has no nonzero eigenvalues.

(iii) if $ker(f) \neq 0$,  and  $f(a)= k1_{A}$  for some
 $0 \neq k \in R$,  then  $r = - k $  is the unique nonzero eigenvalue
of $ad_{f}(a)$,  and  $V_{r}^{ad_{f}(a)} = ker(f)$.

Hence $ad_{f}(a)$ has at most two eigenvalues.
 \label{EigProp1}
\end{prop}

\begin{proof}
(i) If $ker(f) = 0$  then $A_{f}^{2} = 0$ by Proposition
{\ref{NrProp2}}. Therefore  $ad_{f}(a)(x) = 0$ for all $x \in A$.

(ii) If there were $0 \neq k \in R$  and  $0 \neq x \in A$, such
that  $ad_{f}(a)(x) = kx$,  then
\[ kf(x) = f(ad_{f}(a)(x)) = f([a, x]_{f}) = 0,  \]
hence $f(x) = 0$, and
\[ -xf(a) = af(x) - xf(a) = kx. \]
We would have $f(a) = - k1_{A}$,  a contradiction.

(iii) If $ad_{f}(a)(x) = rx$ for some $0 \neq r \in R$  and
  $0 \neq x \in A$, by doing the same thing as in the proof of (ii)
  we have $f(x) = 0$, and
\[ -kx = -xf(a) = af(x) - xf(a) = rx,  \]
hence $r = -k$ and $x \in ker(f)$.  It is also clear that for any
$x \in ker(f)$, we have  $ad_{f}(a)(x) = -xf(a) = rx$.
\wbox
\end{proof}

\subsection{The kernel of an averaging operator}
We end this thesis with a discussion of the kernel of averaging
operators. We know that for an averaging operator $f$ on an
$R$-algebra $A$, $[A, A]_f \subseteq ker(f)$. It is not true in
general that $ker(f) = [A,A]_f$ for an averaging operator $f$ on
an $R$-algebra $A$.

\begin{ex}
Let $R = {\bf{Z}}$,  the ring of integers and let $A = {\bf{Z}}/6
{\bf{Z}}$.  Define $f(a) = 2a$ for  $a \in A$. $f$ is an averaging
operator. We have $[A, A]_f = 0$, but $ker(f) \neq 0$.
\end{ex}

We want to know when $ker(f) = [A,A]_f$ holds.

\begin{prop}
If $(A, f)$ is a unitary averaging algebra, then $ker(f) =
[A,A]_f$.
\end{prop}
\begin{proof}
Let $a \in ker(f)$, then $ a = af(1_A) - 1_Af(a) = [a, 1_A]_f \in
[A,A]_f$. \wbox
\end{proof}

\begin{prop}
Let $R$ be a field, $f$ be an averaging operator on an $R$-algebra
$A$. If  $f(1_A)$ is not a zero divisor and f(A) is of finite
dimension,  then $ker(f) = [A,A]_f$ holds.
\end{prop}
\begin{proof}
Let $a_1, ..., a_k$ be a basis of $f(A)$, where $k = dim(f(A))$.
Note that $f(1_R)a_1, ..., f(1_R)a_k$ is also a basis of $f(A)$,
since $f(1_R)$ is not a zero divisor. There exist $r_1, ... , r_k
\in R$ such that
\[f(1_R) = \sum_{ 1 \leq i \leq k}r_if(1_R)a_i = f(1_R)\sum_{ 1 \leq i \leq
k}r_ia_i.\] \noindent Therefore $ 1_A = \sum_{ 1 \leq i \leq
k}r_ia_i \in f(A)$. There exist $b \in A$ such that $1_A = f(b)$.
Take $ a \in ker(f)$, then $ a = af(b)-bf(a) = [a, b]_f \in [A,
A]_f$. \wbox
\end{proof}

\begin{prop}
\label{prop:ker1}
Let $X$ be a set, $(F_X, f_X)$ be the free averaging $R$-algebra
on $X$. Then $ker(f_X) = [F_X, F_X]_{f_X}$.
\end{prop}
\begin{proof}
If $X = \emptyset$, then $F_X = R[y_{1_R}]$ and $f_X(u) =
y_{1_R}u$ for all $u \in F_X$. We have
  \[ ker(f_X) = {0} = [F_X, F_X]_{f_X}. \]
Now consider the case where $X \neq \emptyset$.
 We only need to show that
 \[ker(f_X) \subseteq [F_X, F_X]_{f_X}. \]

 Let $\displaystyle w = \sum_{1 \leq i \leq k}r_iw_i$ be an nonzero element of
 $ker(f_X)$, where $k > 0$, $w_i$ are distinct monomials of $F_X$
 with coefficient $1_R$
and all $r_i$ are nonzero elements of $R$.

\noindent
{\bf Step 1. }
We first assume that
 $f_X(w_1) = f_X(w_2) = ... = f_X(w_k)$.
Note that under this assumption,  we have $\Sigma_ir_i = 0$, since
$f_X(w) = (\Sigma_ir_i)f_X(w_1) = 0$ and $f_X(w_1) \neq 0$.  All
$u_i$ are distinct, and all $v_i$ are distinct.
 We will show that $w \in [F_X, F_X]_{f_X}$.

 $k = 1$ is an impossible case, according to the definition of
 $f_X$.

 If $ k = 2$, then $ w = r_1u_1v_1 + r_2u_2v_2$, where $u_1, u_2
 \in \Theta(X)$, and $v_1, v_2 \in R[Y]$.
 We have  $r_2 = - r_1$ and $v_1 \neq v_2$. At least one of
 $ v_1$ and $v_2$ is not $1_R$. Assume that $v_1 =
 y_{\theta_1}...y_{\theta_p}$  for some $\theta_1, ..., \theta_p
 \in \Theta(X)$, $p > 0$. Then we have
 \begin{eqnarray*}
&& f_X(w) = r_1f_X(u_1v_1) - r_1f_X(u_2v_2)
\\ && = r_1y_{u_1}y_{\theta_1}...y_{\theta_p} - r_1y_{u_2}v_2 = 0.
 \end{eqnarray*}
 Therefore, without loss of generality (reorder  $y_{\theta_1},...,y_{\theta_p}$
 if necessary) , we can assume  $y_{\theta_1} = y_{u_2}$. Then
  \[\theta_1 = u_2,\]
 \[v_2 = y_{u_1}y_{\theta_2}...y_{\theta_p}.\]
Hence we have
 \begin{eqnarray*}
 &&  w = r_1(u_1v_1 - u_2v_2)
\\ && = r_1(u_1y_{\theta_1}...y_{\theta_p} -
\theta_1y_{u_1}y_{\theta_2}...y_{\theta_p})
 \\ && = r_1(u_1f_X(\theta_1y_{\theta_2}...y_{\theta_p}) -
 \theta_1y_{\theta_2}...y_{\theta_p}f_X(u_1))
 \\&& = [r_1u_1, \theta_1y_{\theta_2}...y_{\theta_p}]_{f_X}.
 \end{eqnarray*}
 Therefore $w \in [F_X, F_X]_{f_X}$.

If $k > 2$, then $r_k = -r_1 - r_2-...-r_{k-1}$, and
\[ w = r_1(w_1 - w_k) + r_2(w_2 - w_k) - ... - r_{k-1}(w_{k-1} -
w_k) .\]

According to the discussion for the case $k = 2$, we know for each
$i$, $1 \leq i \leq k-1$, We have $r_i(w_i - w_k) \in [F_X,
F_X]_{f_X}$, hence $ w \in [F_X, F_X]_{f_X}$.

\noindent
{\bf Step 2. }
Note that, by the definition of $f_X$, each $f_X(w_i)$
is a monomial of $F_X$ with coefficient $1_R$.
By rearranging $w_i$, we can assume that there is a
partition
\[ \{1,\ldots,n_1,n_1+1,\ldots,n_2,\ldots,
    n_{r-1}+1,\ldots,n_r=k\}\]
of $\{1\ldots,k\}$ such that
\[ f(w_i)=u_{j+1},\ {\rm\ for\ } n_{j}+1\leq i\leq n_{j+1},
    0\leq j\leq r-1 \]
(taking $n_0=0$)
and $u_1,\ldots,u_r$ are distinct monomials of
$F_X$ with coefficient $1_R$.
Then since $u_1,\ldots,u_r$ are distinct, they are
linearly independent. From $f_X(w)=0$, we get
\begin{eqnarray*}
\lefteqn{ 0 = \sum_{j=0}^{r-1} \sum_{i=n_j+1}^{n_{j+1}}
    r_i f(w_i)}\\
&=&\sum_{j=0}^{r-1} \sum_{i=n_j+1}^{n_{j+1}}
    r_i u_{j+1}.
\end{eqnarray*}
So
\[ \sum_{i=n_j+1}^{n_{j+1}} r_i u_{j+1}=0. \]
So
\[ \sum_{i=n_j+1}^{n_{j+1}} r_i f(w_i)=0 ,\ j=0,\ldots,r-1\]
with $f(w_{n_j+1})=\ldots = f(w_{n_{j+1}})$. Now by Step 1,
$\sum_{i=n_j+1}^{n_{j+1}} r_i w_i\in [F_X, F_X]_{f_X}$ for each
$j=0,\ldots,r-1$. Therefore, $w$ is in $[F_X, F_X]_{f_X}$. \wbox
\end{proof}

Let $S = \{s_j : j \in J\}$ be a subset of an averaging
$R$-algebra $(A, f)$, such that $S$ generates $(A,f)$ ( which
means the only averaging subalgebra of $(A,f)$ containing $S$ is
itself). Let $X = \{x_j : j \in J\}$ be a set. Define a map $\eta
: X \rightarrow S $ such that $\eta(x_j) = s_j$ for all $j \in J$.
There exists a unique averaging homomorphism $\varphi : (F_X, f_X)
\rightarrow (A,f)$ such that $\varphi \circ i_X = \eta$. We have
the following commutative diagram:
\[\begin{array}{ccc}
F_X & \ola{f_X} & F_X \\  \downarrow {\varphi} && \downarrow
{\varphi}\\ A & \ola {f} & A.
\end{array} \]

Note that $\varphi$ is surjective, since its image $\varphi(F_X)$
is an averaging subalgebra of $A$, and $\varphi(F_X)$ contains
$S$.
\begin{lemma}
\label{lem:ker}
With the notations above, the following are true.

(i)  $ker(f_X) \subseteq \varphi^{-1}(ker(f))$.

(ii)  $ ker(\varphi) \subseteq \varphi^{-1}(ker(f))$.

(iii)  $ker(f_X) + ker(\varphi) \subseteq \varphi^{-1}(ker(f))$.

(iv)  $f_X(ker(\varphi)) \subseteq ker(\varphi) \cap f_X(F_X)$.
\end{lemma}
\begin{proof}
(i)
\[
 ker(f_X) \subseteq ker(\varphi\circ f_X)
    =\ker (f\circ \varphi) = \varphi^{-1} (ker (f)).\]

(ii) Let $u \in ker(\varphi)$. Then $\varphi(u) = 0 \in ker(f)$.

(iii) Let $w \in ker(f_X)$, $u \in ker(\varphi)$, then
      $\varphi(w + u) = \varphi(w) + \varphi(u) = \varphi(w) \in
      ker(f)$. This proves (iii).

(iv)  Let   $u \in ker(\varphi)$. Since $\varphi(f_X(u)) =
f(\varphi(u)) = f(0) = 0$, $f_X(u) \in ker(\varphi)$. Hence
$f_X(ker(\varphi)) \subseteq ker(\varphi)$. $f_X(ker(\varphi))
\subseteq f_X(F_X) $ is clearly true.  \wbox
\end{proof}

\begin{prop}
\label{prop:ker2}
With the notations above, the following are equivalent.

(i) $ker(f) = [A,A]_f$.

(ii) $\varphi^{-1}(ker(f)) = ker(f_X) + ker(\varphi)$.

(iii) $ker(\varphi) \cap f_X(F_X) = f_X(ker(\varphi))$.

\end{prop}

\begin{proof}
(i) $\Rightarrow$ (ii).
By Lemma~\ref{lem:ker},
we only need to show that
$\varphi^{-1}(ker(f)) \subseteq ker(f_X) + ker(\varphi)$.
 Let $w \in
\varphi^{-1}(ker(f))$, then $\varphi(w) \in ker(f)$. There exist
$a_i, b_i \in A$ such that $\varphi(w) = \Sigma_i[a_i, b_i]_f$.
Since $\varphi$ is surjective, there exist $s_i, t_i \in F_X$ such
that $\varphi(u_i) = a_i$ and $\varphi(v_i) = b_i$. Then
$\varphi(w) = \varphi(\Sigma_i[s_i, t_i]_{f_X})$. Let $u =
\Sigma_i[s_i, t_i]_{f_X}$, $ v = w - u$. We have $w = u + v$, $ u
\in ker(f_X)$, and $v \in ker(\varphi)$.

(ii) $\Rightarrow$ (i). It is clear that $[A,A]_f \subseteq
ker(f)$.  Let $a \in ker(f)$. There exist $w \in
\varphi^{-1}(ker(f))$, such that $a = \varphi(w)$. Also there
exist $u \in ker(f_X)$, $v \in ker(\varphi)$, such that $ w = u +
v$. So we have $a = \varphi(w) = \varphi(u) + \varphi(v) =
\varphi(u)$. Since $ u \in ker(f_X) = [F_X, F_X]_{f_X}$, there
exist $s_i, t_i \in F_X$, such that $u = \Sigma_i[s_i,
t_i]_{f_X}$. Hence $a = \varphi(u) = \Sigma_i[\varphi(s_i),
\varphi(t_i)]_f \in [A, A]_f$.

(ii) $\Rightarrow$ (iii). We only need to show $ker(\varphi)
\cap f_X(F_X) \subseteq  f_X(ker(\varphi))$.

Let $w \in ker(\varphi) \cap f_X(F_X)$. Then $w = f_X(w')$ for
some $w' \in F_X$. We have $f(\varphi(w')) = \varphi(f_X(w')) =
\varphi(w) = 0$. This implies that $w' \in \varphi^{-1}(ker(f))$.
Therefore there exist $ u \in ker(f_X)$ and $v \in ker(\varphi)$,
such that $ w' = u + v$. Hence
  $w = f_X(w') = f_X(u) + f_X(v) = f_X(v) \in f_X(ker(\varphi))$.

(iii) $\Rightarrow$ (ii). We only need to show that
$\varphi^{-1}(ker(f)) \subseteq ker(f_X) + ker(\varphi)$.

Let $w \in \varphi^{-1}(ker(f))$, then  $\varphi(f_X(w)) =
f(\varphi(w))  = 0$. This means $ f_X(w) \in ker(\varphi) \cap
f_X(F_X) = f_X(ker(\varphi))$. Therefore $f_X(w) = f_X(u)$ for
some $ u \in ker(\varphi)$. Let $v = w - u$, then $v \in ker(f)$
and $ w = u + v \in ker(f_X) + ker(\varphi)$. \wbox

\end{proof}

 \addcontentsline{toc}{section}{\numberline {}References}

\end {document}